\documentclass{amsart}
\usepackage{amssymb}
\usepackage{amsmath}
\usepackage{amsfonts}

\newtheorem{theorem}{Theorem}
\theoremstyle{plain}

\newtheorem{corollary}{Corollary}

\newtheorem{definition}{Definition}
\newtheorem{example}{Example}

\newtheorem{lemma}{Lemma}

\newtheorem{proposition}{Proposition}
\newtheorem{remark}{Remark}

\numberwithin{equation}{section}
\input{tcilatex}

\begin{document}
\title[Local Cohomology]{Local Cohomology and non commutative Gorenstein
algebras}
\author{Roberto Mart\'{\i}nez-Villa}
\address{Centro de Ciencias Matem\'{a}ticas, UNAM Morelia}
\email{mvilla@matmor.unam.mx}
\urladdr{http://www.matmor.unam.mx}
\thanks{}
\date{October 22, 2012}
\subjclass[2000]{Primary 15A15; Secondary 05A15, 15A18}
\keywords{Gorenstein, local cohomology, non commutative}

\begin{abstract}
In this paper we continue the study of non connected graded Gorenstein
algebras initiated in [13], the main result is the proof of a version of the
Local Cohomology formula.
\end{abstract}

\maketitle

\section{\protect\bigskip Introduction}

\bigskip Non commutative versions of regular and Gorenstein algebras have
been studied by several authors, [2], [3],[6],[7],[8],[9], [19], they
usually deal with graded connected algebras. In [11],[15 ] we studied non
connected Artin Schelter regular algebras and proved that for such algebras
the Local Cohomology formula and Serre duality hold. A natural example of
such algebras is the preprojective $\Bbbk $-algebra of an Euclidean Diagram $%
Q$ [10].

In [13] we investigated non connected graded Artin Schelter Gorenstein
algebras, an easy example is the following: let $\Gamma $ be a preprojective 
$\Bbbk $-algebra of an Euclidean Diagram $Q$ with only sinks and sources and 
$\Lambda =\Bbbk Q\triangleright D(\Bbbk Q)$ the trivial extension. Then the
algebra $\Lambda \otimes _{\Bbbk }\Gamma $ is non connected and Artin
Schelter Gorenstein.

\bigskip

The aim of the paper is to continue the study of these algebras and to
provide a non connected version of the Local Cohomology formula. The article
consists of two sections; in the first one we fix the notation and recall
from [11], [12], [13] some basic results on Artin Schelter regular and
Gorenstein algebras, then we study the structure of a generalization of non
connected AS Gorenstein algebras. In the second section we give an
elementary proof of the Local Cohomology formula for such algebras.\bigskip 

{\large Thanks: I want to express my gratitude to Jun-ichi Miyachi for his
criticisms and some valuable suggestions.}

\section{Definitions and basic results}

We recall first some basic definitions and results on graded Gorenstein
algebras, then we concentrate in the study of the structure of graded Artin
Schelter Gorenstein algebras. For further properties on graded and
Gorenstein algebras we refer the reader to [1], [9],[12],[14].

\begin{definition}
Let $\Bbbk $ be a field, a locally finite positively graded $\Bbbk $-algebra
is a positively graded $\Bbbk $-algebra $\Lambda =\underset{i\geq 0}{\oplus }%
\Lambda _{i}$ such that:

i) $\Lambda _{0}=\Bbbk \times \Bbbk \times ...\Bbbk $

ii) dim$_{\Bbbk }\Lambda _{i}<\infty .$
\end{definition}

\begin{example}
Given a finite quiver $Q$ and a field $\Bbbk $, the quiver algebra is graded
by path length, given an homogeneous ideal $I$ the quotient $\Bbbk Q/I $ is
a locally finite positively graded $\Bbbk $-algebra.
\end{example}

The graded algebras we will consider here will be always locally finite
positively graded $\Bbbk $-algebras.

\begin{definition}
Given a $%
\mathbb{Z}
$-graded module $M=\{M_{i}\}_{i\in 
\mathbb{Z}
}$ over a locally finite positively graded $\Bbbk $-algebra $\Lambda $, we
say $M$ is locally finite if $\dim _{\Bbbk }M_{{}}<\infty $ for all $i$.
\end{definition}

Given a $%
\mathbb{Z}
$-graded module $M=\{M_{i}\}_{i\in 
\mathbb{Z}
}$ we denote by $M[n]$ the $n$-th shift defined by $M[n]$ $_{i}=M_{n+i}$ and
we denote the $n$-truncation of $M$ by $M_{\geq n}$, where $M_{\geq n}$ is
defined by $(M_{\geq n})_{j}=\left\{ 
\begin{array}{ccc}
\text{0} & \text{if} & \text{j}<n \\ 
\text{M}_{j} & \text{if} & \text{j}\geq \text{n}%
\end{array}%
\right. $

\begin{definition}
Given a positively graded $\Bbbk $-algebra $\Lambda =\underset{i\geq 0}{%
\oplus }\Lambda $, the graded Jacobson radical is $m=\Lambda _{\geq 1}$.
\end{definition}

We denote by $Gr_{\Lambda }$ the category of graded $\Lambda $-modules and
degree zero maps. Given graded $\Lambda $-modules $M$ and $N$, $Hom_{\Lambda
}(M,N)_{k}$ is the set of all maps $f:M\rightarrow N$ such that $%
f(M_{j})\subseteq N_{j+k}$ and we call them, maps in degree $k$. By $%
Hom_{\Lambda }(M,N)$ we mean $Hom_{\Lambda }(M,N)=\underset{k\in 
\mathbb{Z}
}{\oplus }Hom_{\Lambda }(M,N)_{k}$. The maps in all degrees $Hom_{\Lambda
}(M,N)$ is a graded $\Bbbk $-vector space. We have isomorphisms: $%
Hom_{\Lambda }$($M,N$)$_{k}\cong Hom_{\Lambda }$($M,N[k]$)$_{0}$\linebreak $%
\cong Hom_{\Lambda }(M[-k],N)_{0}$. If $M$ is finitely generated and $N$
locally finite, then \linebreak $Hom_{\Lambda }(M,N)$ is locally finite.

In a similar way the $k$ extensions in degree zero, $Ext_{\Lambda
}^{k}(-,?)_{0}$ are the derived functors of $Hom_{\Lambda }(-,?)_{0}$. We
define $Ext_{\Lambda }^{k}(M,N)_{n}=Ext_{\Lambda }^{k}(M,N[n])_{0}$ and $%
Ext_{\Lambda }^{k}(M,N)$\linebreak $=\underset{n\in 
\mathbb{Z}
}{\oplus }Ext_{\Lambda }^{k}(M,N)_{n}.$

Denote by $l.f$.$Gr_{\Lambda }$ the full subcategory of $Gr_{\Lambda }$
consisting of all locally finite modules. Then $l.f$.$Gr_{\Lambda }$ is
abelian and there is a duality $D:$ $l.f$.$Gr_{\Lambda }\rightarrow $ $l.f$.$%
Gr_{\Lambda ^{op}}$ given by $D(M)_{j}=Hom_{\Bbbk }(M_{-j},\Bbbk )$ and $%
D(M)=\{D(M)_{j}\}_{j\in 
\mathbb{Z}
}$.

A notion of regular was introduced for non commutative connected positively
graded algebras by Artin and Schelter [1] a slight generalization is the
following: (See [11],[12])

\begin{definition}
Let $\Lambda $ be a locally finite positively graded $\Bbbk $-algebra. Then $%
\Lambda $ is called Artin-Schelter regular if the following conditions are
satisfied:

i) There is an integer $n$ such that all graded simple have projective
dimension $n$.

ii) For any graded simple $S$ and an integer $0\leq i<n$, $Ext_{\Lambda
}^{i}(S$, $\Lambda )=0.$

iii) The assignment $S\rightarrow Ext_{\Lambda }^{n}(S$, $\Lambda )$ gives a
bijection between the graded simple $\Lambda $- modules and the graded $%
\Lambda ^{op}$-graded simple modules.
\end{definition}

The above definition was extended to graded categories in [15].

\begin{definition}
A ring $R$ is called Gorenstein if $R$ has finite injective dimension both
as a left and as a right $R$ -module. We denote the left (right) injective
dimension by $inj\dim _{R}R$ ( $inj\dim R_{R}).$
\end{definition}

It was proved in [18] that Gorenstein implies $inj\dim _{R}R$ =$inj\dim
R_{R} $, but it is not known whether a one side condition implies the
condition on both sides.

The notion of Artin-Schelter regular inspired for connected graded algebras
a definition of Artin Schelter Gorenstein (AS\ Gorenstein, for short) that
has been used by several authors like: [8],[9], [17], [19]. We will use here
the following variation of that definition for non connected graded algebras:

\begin{definition}
Let $\Bbbk $ be a field and $\Lambda $ a locally finite positively graded $%
\Bbbk $-algebra. Then we say that $\Lambda $ is graded Artin Schelter
Gorenstein if the following conditions are satisfied:

There exists a non negative integer $n$, called the graded injective
dimension of $\Lambda $, such that:

i) For all graded simple $S_{i}$ concentrated in degree zero and non
negative integers $j\neq n$, Ext$_{\Lambda }^{j}(S_{i},\Lambda )=0$.

ii) We have an equality Ext$_{\Lambda }^{n}(S_{i},\Lambda )=S_{i}^{\prime
}[-n_{i}]$, with $S_{i}^{\prime }$ a graded $\Lambda ^{op}$-simple.

iii) For a non negative integer $k\neq n,$ Ext$_{\Lambda
^{op}}^{k}(Ext_{\Lambda }^{n}(S_{i},\Lambda ),\Lambda )=0$ and\linebreak\ Ext%
$_{\Lambda ^{op}}^{n}(Ext_{\Lambda }^{n}(S_{i},\Lambda ),\Lambda )=S_{i}$.
\end{definition}

\bigskip\ Since it is not clear that condition i) implies $\Lambda $ of
finite injective dimension, observe that in this definition we are not
assuming the algebra is Gorenstein, however the way we defined graded AS
Gorenstein is a two sided condition. Would be interesting to know if graded
AS Gorenstein implies Gorenstein. We will some times assume that an algebra
is both Gorenstein and graded AS Gorenstein.

We next recall a result from [13] that has an interesting corollary related
to the above remark.

\bigskip

\begin{theorem}
Let $R$ be an arbitrary Gorenstein ring and $M$ a left $R$-module with a
projective resolution consisting of finitely generated projective modules.
Assume there is a non negative integer $n$ such that $Ext_{R}^{j}(M,R)=0$
for $j\neq n.$ Then $Ext_{R}^{n}(M,R)$ satisfies the following conditions:

a) $Ext_{R^{op}}^{j}(Ext_{R}^{n}(M,R),R)=0$ for $j\neq n.$

b) $Ext_{R^{op}}^{n}(Ext_{R}^{n}(M,R),R)\cong M.$

In case $R$ is graded and $M$ a graded $R$-module, the isomorphism is as
graded $R$-modules.
\end{theorem}

\begin{proof}
Assume $M$ is of finite projective dimension $n$.

Let $n=0$. Then $M$ is projective, $M^{\ast }$ is projective and $M^{\ast
\ast }\cong M$ and $Ext_{R}^{\ell }(M,R)=0$ for $\ell \neq 0$.

Assume $pdM=n>0$ and let: $0\rightarrow P_{n}\rightarrow P_{n-1}\rightarrow
...P_{1}\rightarrow P_{0}\rightarrow M\rightarrow 0$ be a projective
resolution of minimal length with $P_{j}$ finitely generated for all $j$.

Dualizing with respect to the ring we obtain an exact sequence:

$0\rightarrow P_{0}^{\ast }\rightarrow P_{1}^{\ast }\rightarrow
...P_{n-1}^{\ast }\rightarrow P_{n}^{\ast }\rightarrow
Ext_{R}^{n}(M,R)\rightarrow 0.$

Dualizing again the complex: $0\rightarrow P_{n}^{\ast \ast }\rightarrow
P_{n-1}^{\ast \ast }\rightarrow ...P_{1}^{\ast \ast }\rightarrow P_{0}^{\ast
\ast }\rightarrow 0$ is isomorphic to the complex:

$0\rightarrow P_{n}\rightarrow P_{n-1}\rightarrow ...P_{1}\rightarrow
P_{0}\rightarrow 0.$

It follows $Ext_{R^{op}}^{i}(Ext_{R}^{n}(M,R),R)=0$ for $i\neq n$ and $%
Ext_{R^{op}}^{n}(Ext_{R}^{n}(M,R),R)\cong M$.

Assume $M$ is of infinite projective dimension. Let *) ...$\rightarrow
P_{n}\rightarrow P_{n-1}\rightarrow ...P_{1}\rightarrow P_{0}\rightarrow
M\rightarrow 0$ be a projective resolution with $P_{j}$ finitely generated
for all $j$.

Consider first the case $n=0$.

Dualizing with respect to the ring we obtain an exact sequence:

$0\rightarrow M^{\ast }\rightarrow P_{0}^{\ast }\rightarrow P_{1}^{\ast
}\rightarrow ...P_{t}^{\ast }\rightarrow P_{t+1}^{\ast }\rightarrow
Y\rightarrow 0$, where $t$ is the injective dimension of $R$ .

Then $M^{\ast }=\Omega ^{t+2}Y$ and for $i>0$, $Ext_{R}^{i}(M^{\ast
},R)=Ext_{R}^{i}(\Omega ^{t+2}Y,R)=$

$Ext_{R}^{i+t+1}(Y,R)=0$, also $Ext_{R}^{1}(\Omega
^{t}Y,R)=Ext_{R}^{t+1}(Y,R)=0$.

Hence the exact sequences: $0\rightarrow M^{\ast }\rightarrow P_{0}^{\ast
}\rightarrow \Omega ^{t+1}Y\rightarrow 0$ and $0\rightarrow \Omega
^{t+1}Y\rightarrow P_{1}^{\ast }\rightarrow \Omega ^{t}Y\rightarrow 0$
induce exact sequences:

$0\rightarrow (\Omega ^{t+1}Y)^{\ast }\rightarrow P_{0}^{\ast \ast
}\rightarrow M^{\ast \ast }\rightarrow 0$ and $0\rightarrow (\Omega
^{t}Y)^{\ast }\rightarrow P_{1}^{\ast \ast }\rightarrow (\Omega
^{t+1}Y)^{\ast }\rightarrow 0$.

We have proved the sequence: $P_{1}^{\ast \ast }\rightarrow P_{0}^{\ast \ast
}\rightarrow M^{\ast \ast }\rightarrow 0$ is exact.

It follows $M\cong M^{\ast \ast }$ and $Ext_{R^{op}}^{i}(M^{\ast },R)=0$ for 
$i\neq n$.

Assume now $n>0$. Dualizing $\ast )$ with respect to the ring we get the
complex:

$0\rightarrow M^{\ast }\rightarrow P_{0}^{\ast }\rightarrow P_{1}^{\ast
}\rightarrow ...P_{n-1}^{\ast }\overset{f_{n}^{\ast }}{\rightarrow }%
P_{n}^{\ast }\overset{f_{n\ast 1}^{\ast }}{\rightarrow }P_{n+1}^{\ast }%
\overset{f_{n+2}^{\ast }}{\rightarrow }P_{n+2}^{\ast }\rightarrow ...$

whose homology is zero except at degree $n$, where $K\func{erf}_{n+1}^{\ast
}/\func{Im}f_{n}^{\ast }=Ext_{R}^{n}(M,R)$.

Let $C=P_{n}^{\ast }/\func{Im}f_{n}^{\ast }$ and $X=\func{Im}f_{n+1}^{\ast
}=K\func{erf}_{n+2}^{\ast }=P_{n}^{\ast }/K\func{erf}_{n+1}^{\ast }$.

There is an exact sequence: *) $0\rightarrow Ext_{R}^{n}(M,R)\rightarrow
C\rightarrow X\rightarrow 0$.

Consider the exact sequence:

$0\rightarrow X\rightarrow P_{n+1}^{\ast }\rightarrow P_{n+2}^{\ast
}\rightarrow ...P_{n+t-1}^{\ast }\rightarrow P_{n+t}^{\ast }\rightarrow
Y\rightarrow 0$, with $t$ the injective dimension of $R$.

Then for $i>0$, we have isomorphisms:

$Ext_{R^{op}}^{i}(X,R)=Ext_{R^{op}}^{i}(\Omega
^{t}Y,R)=Ext_{R^{op}}^{i+t}(Y,R)=0.$

By the sequence $\ast )$ and the long homology sequence we have an exact
sequence:

$Ext_{R^{op}}^{i}(X,R)\rightarrow Ext_{R^{op}}^{i}(C,R)\rightarrow
Ext_{R^{op}}^{i}(Ext_{R}^{n}(M,R),R)\rightarrow Ext_{R^{op}}^{i+1}(X,R)$

It follows that for each $i\geq 1$ there is an isomorphism:

$Ext_{R^{op}}^{i}(C,R)\cong Ext_{R^{op}}^{i}(Ext_{R}^{n}(M,R),R)$ and that
the sequence $0\rightarrow X^{\ast }\rightarrow C^{\ast }\rightarrow
Ext_{R}^{n}(M,R)^{\ast }\rightarrow 0$ is exact.

Since $M^{\ast }=0,$the sequence: $0\rightarrow P_{0}^{\ast }\rightarrow
P_{1}^{\ast }\rightarrow ...P_{n-1}^{\ast }\overset{f_{n}^{\ast }}{%
\rightarrow }P_{n}^{\ast }\overset{f_{n\ast 1}^{\ast }}{\rightarrow }%
C\rightarrow 0$ is exact and $pdC\leq n.$

Being the complexes: $0\rightarrow P_{n}\rightarrow P_{n-1}\rightarrow
...P_{1}\rightarrow P_{0}\rightarrow 0$ and $0\rightarrow P_{n}^{\ast \ast
}\rightarrow P_{n-1}^{\ast \ast }\rightarrow ...P_{1}^{\ast \ast
}\rightarrow P_{0}^{\ast \ast }\rightarrow 0$ isomorphic, it follows $%
Ext_{R^{op}}^{i}(C,R)=0$ for $i\neq n$, $i\neq 0$ and $%
Ext_{R^{op}}^{n}(C,R)=M$.

By the above observations, $Ext_{R^{op}}^{i}(Ext_{R}^{n}(M,R),R)=0$ for $%
i\neq 0$, $i\neq n$ and $Ext_{R^{op}}^{n}(Ext_{R}^{n}(M,R),R)\cong M$.

Consider the following diagram with exact column and first row:

$%
\begin{array}{cccccc}
&  & 0 &  &  &  \\ 
&  & \downarrow &  &  &  \\ 
P_{n-1}^{\ast } & \rightarrow & Kerf_{n\ast 1}^{\ast } & \rightarrow & 
Ext_{R}^{n}(M,R) & \rightarrow 0 \\ 
\downarrow 1 &  & \downarrow &  &  &  \\ 
P_{n-1}^{\ast } & \overset{f_{n}^{\ast }}{\rightarrow } & P_{n}^{\ast } & 
\overset{f_{n\ast 1}^{\ast }}{\rightarrow } & P_{n+1}^{\ast } &  \\ 
&  & \downarrow & \nearrow &  &  \\ 
&  & \func{Im}f_{n\ast 1}^{\ast } &  &  &  \\ 
&  & \downarrow &  &  &  \\ 
&  & 0 &  &  & 
\end{array}%
$

Using the fact $Ext_{R}^{1}(X,R)=Ext_{R}^{1}(\func{Im}f_{n\ast 1}^{\ast
},R)=0$ we obtain by dualizing an exact commutative diagram:

$%
\begin{array}{ccccccc}
&  &  & 0 &  &  &  \\ 
&  &  & \downarrow &  &  &  \\ 
&  &  & (\func{Im}f_{n\ast 1}^{\ast })^{\ast } &  &  &  \\ 
&  & \nearrow & \downarrow &  &  &  \\ 
& P_{n+1}^{\ast \ast } & \overset{f_{n+1}^{\ast \ast }}{\rightarrow } & 
P_{n}^{\ast \ast } & \overset{f_{n}^{\ast \ast }}{\rightarrow } & 
P_{n+1}^{\ast \ast } &  \\ 
& \downarrow t &  & \downarrow p &  & \downarrow 1 &  \\ 
0\rightarrow & Ext_{R}^{n}(M,R)^{\ast } & \overset{s}{\rightarrow } & 
(Kerf_{n\ast 1}^{\ast })^{\ast } & \longrightarrow & P_{n+1}^{\ast \ast } & 
\\ 
&  &  & \downarrow &  &  &  \\ 
&  &  & 0 &  &  & 
\end{array}%
$

By Five's lemma, $t$ is an epimorphism and $st=pf_{n+1}^{\ast \ast }=0$
implies $Ext_{R}^{n}(M,R)^{\ast }=0$, as claimed.
\end{proof}

\begin{remark}
Using spectral sequences and results from [], as suggested by Jun-Ichi
Miyachi it is possible to give an alternative proof of the above theorem.
\end{remark}

\begin{corollary}
Let $\Lambda $ be a locally finite positively graded $\Bbbk $-algebra that
is Gorenstein of injective dimension $n$, such that all graded left and all
graded right simple have projective resolutions consisting of finitely
generated projective modules.

Assume the following conditions hold:

a) For all graded simple $S$ and non negative integers $i\neq n$, $%
Ext_{\Lambda }^{i}(S,\Lambda )=0$.

b) Each right module $Ext_{\Lambda }^{n}(S,\Lambda )$ is graded simple.

Then $\Lambda $ is graded AS Gorenstein.
\end{corollary}

We comeback now to the more general notion of AS Gorenstein graded algebra
and prove that, like in the Artin-Schelter regular case (See [15]), the
functor $Ext_{\Lambda }^{n}(-,\Lambda )$ induces a duality between the
categories of graded $\Lambda $-modules of finite length and the
corresponding category of $\Lambda ^{op}$-modules.

\begin{proposition}
Let $\Lambda $ be a graded AS Gorenstein algebra of graded injective
dimension $n$, and assume all graded simple left modules have projective
resolutions consisting of finitely generated projective modules. Then for
any graded left $\Lambda $-module $M$ of finite length, the following is
true:

i) For any non negative integers $i\neq n$, $Ext_{\Lambda }^{i}(M,\Lambda )$=%
$0$.

ii) The right $\Lambda $-module $Ext_{\Lambda }^{n}(M,\Lambda )$ has finite
length and $\ell (Ext_{\Lambda }^{n}(M,\Lambda ))$=$\ell (M)$.

iii) $Ext_{\Lambda ^{op}}^{n}(Ext_{\Lambda }^{n}(M,\Lambda ),\Lambda )\cong
M $.
\end{proposition}

\begin{proof}
We consider first the case $n=0$.

This means that for any graded simple $S$ the dual with respect to the ring $%
S^{\ast }$ is simple, $S\cong S^{\ast \ast }$ and $Ext_{\Lambda
}^{i}(S,\Lambda )=0$ for $i\neq 0$.

We prove the claim by induction on $\ell (M).$

Let $M^{\prime }$ be a maximal graded submodule. Then $M/M^{\prime }=S$ is
simple and the exact sequence: $0\rightarrow M^{\prime }\rightarrow
M\rightarrow S\rightarrow 0$ induces, by hypothesis, the exact sequence: *) $%
0\rightarrow S^{\ast }\rightarrow M^{\ast }\rightarrow (M^{\prime })^{\ast
}\rightarrow 0$. Dualizing again, we have a commutative exact diagram:

$%
\begin{array}{ccccccc}
0\rightarrow & M^{\prime } & \rightarrow & M & \rightarrow & S & \rightarrow
0 \\ 
& \downarrow \theta _{M^{\prime }} &  & \downarrow \theta _{M} &  & 
\downarrow \theta _{S} &  \\ 
0\rightarrow & (M^{\prime })^{\ast \ast } & \rightarrow & M^{\ast \ast } & 
\rightarrow & S^{\ast \ast } & 
\end{array}%
$

By hypothesis $\theta _{M^{\prime }}$ and $\theta _{S}$ are isomorphisms, by
the short Five's lemma $\theta _{M}$ is an isomorphism. By induction
hypothesis $\ell (M^{\prime })=\ell ((M^{\prime })^{\ast })$ and the
exactness of the sequence *) implies $\ell (M)=\ell (M^{\ast })$.

Assume $n>0$ and apply again induction on $\ell (M)$. As before, $M^{\prime
} $ is a maximal graded submodule of $M$ and $M/M^{\prime }=S.$

From the long homology sequence we have an exact sequence:

$Ext_{\Lambda }^{i}(S,\Lambda )\rightarrow Ext_{\Lambda }^{i}(M,\Lambda
)\rightarrow Ext_{\Lambda }^{i}(M^{\prime },\Lambda )$

$Ext_{\Lambda }^{i}(S,\Lambda )=Ext_{\Lambda }^{i}(M^{\prime },\Lambda )=0$
for $i\neq n$ implies $Ext_{\Lambda }^{i}(M,\Lambda )=0$ for $i\neq n$ and $%
\ell (Ext_{\Lambda }^{n}(M^{\prime },\Lambda )=\ell (M^{\prime })$ implies $%
\ell (Ext_{\Lambda }^{n}(M,\Lambda ))=\ell (M).$

We only need to prove $Ext_{\Lambda ^{op}}^{n}(Ext_{\Lambda }^{n}(M,\Lambda
),\Lambda )\cong M$.

Let ...$\rightarrow Q_{i}\rightarrow Q_{i-1}\rightarrow ...Q_{1}\rightarrow
Q_{0}\rightarrow M^{\prime }\rightarrow 0$ and ...$\rightarrow
P_{i}\rightarrow P_{i-1}\rightarrow ...P_{1}\rightarrow P_{0}\rightarrow
S\rightarrow 0$ be graded projective resolutions. By Horseshoe's lemma,
there is an exact commutative diagram:

**) $%
\begin{array}{ccccccc}
& \underset{.}{\overset{.}{.}} &  & \underset{.}{\overset{.}{.}} &  & 
\underset{.}{\overset{.}{.}} &  \\ 
& \downarrow &  & \downarrow &  & \downarrow &  \\ 
0\rightarrow & Q_{i+1} & \rightarrow & Q_{i+1}\oplus P_{i+1} & \rightarrow & 
P_{i+1} & \rightarrow 0 \\ 
& f_{i+1}\downarrow &  & h_{i+1}\downarrow &  & g_{i+1}\downarrow &  \\ 
0\rightarrow & Q_{i} & \rightarrow & Q_{i}\oplus P_{i} & \rightarrow & P_{i}
& \rightarrow 0 \\ 
& \underset{.}{\overset{.}{.}} &  & \underset{.}{\overset{.}{.}} &  & 
\underset{.}{\overset{.}{.}} &  \\ 
& \downarrow &  & \downarrow &  & \downarrow &  \\ 
0\rightarrow & Q_{1} & \rightarrow & Q_{1}\oplus P_{1} & \rightarrow & P_{1}
& \rightarrow 0 \\ 
& f_{1}\downarrow &  & h_{1}\downarrow &  & g_{1}\downarrow &  \\ 
0\rightarrow & Q_{0} & \rightarrow & Q_{0}\oplus P_{0} & \rightarrow & P_{0}
& \rightarrow 0 \\ 
& f_{0}\downarrow &  & h_{0}\downarrow &  & g_{0}\downarrow &  \\ 
0\rightarrow & M^{\prime } & \rightarrow & M & \rightarrow & S & \rightarrow
0 \\ 
& \downarrow &  & \downarrow &  & \downarrow &  \\ 
& 0 &  & 0 &  & 0 & 
\end{array}%
$

Dualizing the diagram **) with respect to the ring we obtain an exact
sequence of complexes:

$%
\begin{array}{ccccccc}
& 0 &  & 0 &  & 0 &  \\ 
& \downarrow &  & \downarrow &  & \downarrow &  \\ 
0\rightarrow & P_{0}^{\ast } & \rightarrow & P_{0}^{\ast }\oplus Q_{0}^{\ast
} & \rightarrow & Q_{0}^{\ast } & \rightarrow 0 \\ 
& g_{1}^{\ast }\downarrow &  & h_{1}^{\ast }\downarrow &  & f_{1}^{\ast
}\downarrow &  \\ 
0\rightarrow & P_{1}^{\ast } & \rightarrow & P_{1}^{\ast }\oplus Q_{1}^{\ast
} & \rightarrow & Q_{1}^{\ast } & \rightarrow 0 \\ 
& \underset{.}{\overset{.}{.}} &  & \underset{.}{\overset{.}{.}} &  & 
\underset{.}{\overset{.}{.}} &  \\ 
& \downarrow &  & \downarrow &  & \downarrow &  \\ 
0\rightarrow & P_{n}^{\ast } & \rightarrow & P_{n}^{\ast }\oplus Q_{n}^{\ast
} & \rightarrow & Q_{n}^{\ast } & \rightarrow 0 \\ 
& g_{n+1}^{\ast }\downarrow &  & h_{n+1}^{\ast }\downarrow &  & 
f_{n+1}^{\ast }\downarrow &  \\ 
0\rightarrow & P_{n+1}^{\ast } & \rightarrow & P_{n+1}^{\ast }\oplus
Q_{n+1}^{\ast } & \rightarrow & Q_{n+1}^{\ast } & \rightarrow 0 \\ 
& \downarrow &  & \downarrow &  & \downarrow &  \\ 
& \underset{.}{\overset{.}{.}} &  & \underset{.}{\overset{.}{.}} &  & 
\underset{.}{\overset{.}{.}} & 
\end{array}%
$

whose homology is zero except at degree $n$. Then we have an exact sequence
of projective resolutions:

$%
\begin{array}{ccccccc}
& 0 &  & 0 &  & 0 &  \\ 
& \downarrow &  & \downarrow &  & \downarrow &  \\ 
0\rightarrow & P_{0}^{\ast } & \rightarrow & P_{0}^{\ast }\oplus Q_{0}^{\ast
} & \rightarrow & Q_{0}^{\ast } & \rightarrow 0 \\ 
& g_{1}^{\ast }\downarrow &  & h_{1}^{\ast }\downarrow &  & f_{1}^{\ast
}\downarrow &  \\ 
0\rightarrow & P_{1}^{\ast } & \rightarrow & P_{1}^{\ast }\oplus Q_{1}^{\ast
} & \rightarrow & Q_{1}^{\ast } & \rightarrow 0 \\ 
& \underset{.}{\overset{.}{.}} &  & \underset{.}{\overset{.}{.}} &  & 
\underset{.}{\overset{.}{.}} &  \\ 
& \downarrow &  & \downarrow &  & \downarrow &  \\ 
0\rightarrow & P_{n}^{\ast } & \rightarrow & P_{n}^{\ast }\oplus Q_{n}^{\ast
} & \rightarrow & Q_{n}^{\ast } & \rightarrow 0 \\ 
& \downarrow &  & \downarrow &  & \downarrow &  \\ 
0\rightarrow & C_{S} & \rightarrow & C_{M} & \rightarrow & C_{M^{\prime }} & 
\rightarrow 0 \\ 
& \downarrow &  & \downarrow &  & \downarrow &  \\ 
& 0 &  & 0 &  & 0 & 
\end{array}%
$

Setting as above $C_{S}=P_{n}^{\ast }/\func{Im}g_{n}^{\ast }$, $%
C_{M}=P_{n}^{\ast }\oplus Q_{n}^{\ast }/\func{Im}h_{n}^{\ast }$ and $%
C_{M^{\prime }}=Q_{n}^{\ast }/\func{Im}f_{n}^{\ast }$ and $X_{S}=P_{n}^{\ast
}/Kerg_{n+1}^{\ast }$, $X_{M}=P_{n}^{\ast }\oplus Q_{n}^{\ast
}/Kerh_{n+1}^{\ast }$ and $X_{M^{\prime }}=Q_{n}^{\ast }/\func{Im}%
f_{n+1}^{\ast }$.

Then there is an exact commutative diagram:

***) $%
\begin{array}{ccccccc}
& 0 &  & 0 &  & 0 &  \\ 
& \downarrow &  & \downarrow &  & \downarrow &  \\ 
0\rightarrow & Ext_{\Lambda }^{n}(S,\Lambda ) & \rightarrow & Ext_{\Lambda
}^{n}(M,\Lambda ) & \rightarrow & Ext_{\Lambda }^{n}(M^{\prime },\Lambda ) & 
\rightarrow 0 \\ 
& \downarrow &  & \downarrow &  & \downarrow &  \\ 
0\rightarrow & C_{S} & \rightarrow & C_{M} & \rightarrow & C_{M^{\prime }} & 
\rightarrow 0 \\ 
& \downarrow &  & \downarrow &  & \downarrow &  \\ 
0\rightarrow & X_{S} & \rightarrow & X_{M} & \rightarrow & X_{M^{\prime }} & 
\rightarrow 0 \\ 
& \downarrow &  & \downarrow &  & \downarrow &  \\ 
& 0 &  & 0 &  & 0 & 
\end{array}%
$

Applying the functor $Ext_{\Lambda ^{op}}^{n}(-,\Lambda )$ to the diagram
***) we have a commutative exact diagram: \newline
$%
\begin{array}{ccc}
\text{M}^{\prime } & \text{M} & \text{S} \\ 
\downarrow \cong & \downarrow \cong & \downarrow \cong \\ 
\text{0}\rightarrow \text{Ext}_{\Lambda }^{n}\text{(C}_{M^{\prime }}\text{,}%
\Lambda \text{)} & \rightarrow \text{Ext}_{\Lambda }^{n}\text{(C}_{M}\text{,}%
\Lambda \text{)} & \rightarrow \text{Ext}_{\Lambda }^{n}\text{(C}_{S}\text{,}%
\Lambda \text{)}\rightarrow \text{0} \\ 
\downarrow \psi _{M^{\prime }} & \downarrow \psi _{M} & \downarrow \psi _{S}
\\ 
\text{0}\rightarrow \text{Ext}_{\Lambda }^{n}\text{(Ext}_{\Lambda }^{n}\text{%
(M}^{\prime }\text{,}\Lambda \text{),}\Lambda \text{)} & \rightarrow \text{%
Ext}_{\Lambda }^{n}\text{(Ext}_{\Lambda }^{n}\text{(M,}\Lambda \text{),}%
\Lambda \text{)} & \rightarrow \text{Ext}_{\Lambda }^{n}\text{(Ext}_{\Lambda
}^{n}\text{(S,}\Lambda \text{),}\Lambda \text{)}\rightarrow \text{0} \\ 
\downarrow & \downarrow & \downarrow \\ 
0 & 0 & 0%
\end{array}%
$

where $\psi _{M^{\prime }}$ and $\psi _{S}$ are isomorphisms. Therefore $%
\psi _{M}$ is an isomorphism.
\end{proof}

\begin{lemma}
Let $\Lambda $ be an algebra over a field $\Bbbk $ and $M$ a finitely
presented left $\Lambda $-module. Then for any left $\Lambda $-module $X$
there is a natural isomorphism of $\Bbbk $-vector spaces: $Hom_{\Bbbk
}(Hom_{\Lambda }(M,X),\Bbbk ))\cong Hom_{\Bbbk }(X,\Bbbk )\otimes _{\Lambda
}M.$
\end{lemma}

\begin{proof}
Let *) $P_{1}\rightarrow P_{0}\rightarrow M\rightarrow 0$ be a finite
projective presentation of $M$. Tensoring with $D(X)=Hom_{\Bbbk }(X,\Bbbk )$%
, we obtain an exact sequence: $D(X)\otimes P_{1}\rightarrow D(X)\otimes
P_{0}\rightarrow D(X)\otimes M\rightarrow 0$ of $\Bbbk $-vector spaces,
which isomorphic to the exact sequence:

**) $Hom_{\Lambda ^{op}}(P_{1}^{\ast }$, $D(X))\rightarrow Hom_{\Lambda
^{op}}(P_{0}^{\ast }$, $D(X))\rightarrow D(X)\otimes M\rightarrow 0$ , $%
P_{i}^{\ast }$ is the dual with respect to the ring.

By adjunction, the sequence **) is isomorphic to the exact sequence:

$Hom_{\Bbbk }(P_{1}^{\ast }\otimes X,\Bbbk )\rightarrow Hom_{\Bbbk
}(P_{0}^{\ast }\otimes X,\Bbbk )\rightarrow D(X)\otimes M\rightarrow 0.$

In addition, $P_{i}^{\ast }\otimes X\cong Hom_{\Lambda }(P_{i},X)$ and **)
is isomorphic to:

$D(Hom_{\Lambda }(P_{1},X))\rightarrow D(Hom_{\Lambda }(P_{0},X))\rightarrow
D(X)\otimes M\rightarrow 0.$

In the other hand, the presentation *) induces an exact sequence:

$0\rightarrow $ $Hom_{\Lambda }(M,X)\rightarrow Hom_{\Lambda
}(P_{0},X)\rightarrow Hom_{\Lambda }(P_{1},X).$

Dualizing it we obtain the exact sequence:

$D(Hom_{\Lambda }(P_{1},X))\rightarrow D(Hom_{\Lambda }(P_{0},X))\rightarrow
D(Hom_{\Lambda }(M,X))\rightarrow 0.$

It follows: $D(X)\otimes M\cong D(Hom_{\Lambda }(M,X))$.
\end{proof}

As a corollary we get:

\begin{proposition}
Let $\Lambda $ be a left coherent algebra over a field $\Bbbk $. Then for
any injective $I$ its dual $Hom_{\Bbbk }(I,\Bbbk )$ is flat.
\end{proposition}

\begin{proof}
Let $\mathfrak{a}$ be a finitely generated left ideal of $\Lambda $, by
definition of coherent, $\mathfrak{a}$ is finitely presented, hence by Lemma
1, there is a natural isomorphism: $D(I)\otimes _{\Lambda }\mathfrak{a}\cong
D(Hom_{\Lambda }(\mathfrak{a},I))$.

The exact sequence $0\rightarrow \mathfrak{a}\overset{j}{\rightarrow }%
\Lambda \rightarrow \Lambda /\mathfrak{a}\rightarrow 0$ induces an exact
sequence:

$0\rightarrow Hom_{\Lambda }(\Lambda /\mathfrak{a},I)\rightarrow
Hom_{\Lambda }(\Lambda ,I)\rightarrow Hom_{\Lambda }(\mathfrak{a}%
,I)\rightarrow 0$. We obtain by dualizing it the exact sequence:

$0\rightarrow D(Hom_{\Lambda }(\mathfrak{a},I))\rightarrow D(Hom_{\Lambda
}(\Lambda ,I))\rightarrow D(Hom_{\Lambda }(\Lambda /\mathfrak{a}%
,I))\rightarrow 0$, which is isomorphic to the exact sequence: $D(I)\otimes
_{\Lambda }\mathfrak{a}\overset{1\otimes j}{\rightarrow }D(I)\otimes
_{\Lambda }\Lambda \rightarrow D(I)\otimes _{\Lambda }\Lambda /\mathfrak{a}%
\rightarrow 0$.

Therefore: $1\otimes j$ is a monomorphism.

It follows by [16] Proposition 3.58, that $D(I)$ is flat.
\end{proof}

The Gorenstein property, or more generally the Cohen Macaulay property, is
related with the existence of a dualizing object, in the commutative case
the ring or another bimodule [4], in other instances a dualizing complex,
for example the injective resolution of the ring [17], [19], in the case of
Artin Schelter regular algebras $\Lambda $ considered in [11], the dualizing
object was a shift of $\Lambda $, in the case we are considering the
dualizing object is also a bimodule, but its description is more subtle, we
dedicate the remain of this section to find its structure.

\begin{proposition}
\lbrack 13] Let $\Lambda $ be a graded AS Gorenstein $\Bbbk $-algebra of
graded injective dimension $n$. Then for any graded simple $S_{j}$
concentrated in degree zero, there exists a unique indecomposable projective 
$Q_{\sigma (j)}$ and a non negative integer $n_{j}$ such that $Ext_{\Lambda
}^{n}(S_{j},Q_{\sigma (j)}[-n_{j}])_{0}\neq 0$ and the assignment $%
S_{j}\rightarrow Q_{\sigma (j)}$ is a bijection.
\end{proposition}

\begin{proof}
We have the following isomorphisms: $Ext_{\Lambda }^{n}(S_{j},\Lambda )=%
\underset{k\in Z}{\oplus }Ext_{\Lambda }^{n}(S_{j},\Lambda )_{k}\cong
S_{j}^{\prime }[-n_{j}]$, where $S_{j}^{\prime }$ has dimension one as $K$%
-vector space. Since the isomorphism is as $\Bbbk $-graded vector spaces, $%
Ext_{\Lambda }^{n}(S_{j},\Lambda )_{k}\neq 0$ implies $k=n_{j}$. The algebra 
$\Lambda $ decomposes in sum of indecomposables $\Lambda \cong \underset{i=1}%
{\overset{m}{\oplus }}Q_{i}$. Then $Ext_{\Lambda }^{n}(S_{j},\Lambda
)_{n_{j}}=\underset{i=1}{\overset{m}{\oplus }}Ext_{\Lambda
}^{n}(S_{j},Q_{i}[n_{j}])_{0}=S_{j}^{\prime }[-n_{j}]$ and there exists a
unique integer $\sigma (j)$ such that $Ext_{\Lambda }^{n}(S_{j},Q_{\sigma
(j)}[n_{j}])_{0}\neq 0$.

We will prove that the function $j\rightarrow \sigma (j)$ is injective,
hence bijective.

Assume for some simple $S_{k}$ there is an isomorphism $Q_{\sigma (j)}\cong
Q_{\sigma (k)}$. But both $Ext_{\Lambda }^{n}(S_{j},Q_{\sigma (j)})\neq 0$
and $Ext_{\Lambda }^{n}(S_{k},Q_{\sigma (k)})\neq 0.$

There are isomorphisms:

$Ext_{\Lambda }^{n}(S_{j},Q_{\sigma (j)})=Ext_{\Lambda }^{n}(S_{j},\Lambda
)\otimes _{\Lambda }Q_{\sigma (j)}\cong S_{j}^{\prime }\otimes _{\Lambda
}Q_{\sigma (j)}[-n_{j}]$ and $Ext_{\Lambda }^{n}(S_{k},Q_{\sigma (k)})$%
\linebreak $=Ext_{\Lambda }^{n}(S_{k},\Lambda )\otimes _{\Lambda }Q_{\sigma
(k)}\cong S_{k}^{\prime }\otimes _{\Lambda }Q_{\sigma (k)}[-n_{k}].$

After dualizing we obtain the following isomorphisms:

$D(S_{j}^{\prime }\otimes _{\Lambda }Q_{\sigma (j)}$[-n$_{j}$])=$%
Hom_{\Lambda ^{op}}(S_{j}^{\prime },D(Q_{\sigma (j)})[n_{j}])\neq $0 and D(S$%
_{k}^{\prime }\otimes _{\Lambda }$Q$_{\sigma (k)}$[-n$_{k}$])\linebreak $%
=Hom_{\Lambda ^{op}}(S_{k}^{\prime },D(Q_{\sigma (k)})[n_{k}])\neq 0$.

Since $Q_{\sigma (j)}\cong Q_{\sigma (k)}$ and $D(Q_{\sigma (j)})\cong
D(Q_{\sigma (k)})$ has simple socle, it follows $S_{j}^{\prime }\cong
S_{k}^{\prime }$, in fact $D(S_{j}^{\prime })\cong Q_{\sigma (j)}/JQ_{\sigma
(j)}\cong S_{\sigma (j)}$.
\end{proof}

We have proved $Ext_{\Lambda }^{n}(S_{j},\Lambda )\cong D(S_{\sigma
(j)})[-n_{j}].$

What we want to do next is to prove that $\overset{m}{\underset{j=1}{\oplus }%
}Q_{\sigma (j)}[n_{j}]$ has a structure of $\Lambda $-$\Lambda $ bimodule
and it is projective both as left and as right $\Lambda $-module.

Let $Q_{k}$ be an indecomposable projective $\Lambda $-module with injective
resolution:

$0\rightarrow Q_{k}\rightarrow I_{0}^{(k)}\rightarrow I_{1}^{(k)}\rightarrow
I_{2}^{(k)}\rightarrow ...I_{n-1}^{(k)}\rightarrow I_{n}^{(k)}\rightarrow
I_{n+1}^{(k)}\rightarrow ...$

Here $I_{n}^{(k)}$ decomposes as $I_{n}^{(k)}=$ $I_{n}^{\prime (k)}\oplus $ $%
I_{n}^{\prime \prime (k)}$, where $I_{n}^{\prime (k)}$ is torsion and $%
I_{n}^{\prime \prime (k)}$ is torsion free in the sense of $[M]$.

Then $0\neq $ $Ext_{\Lambda }^{n}(S_{\sigma ^{-1}(k)}[-n_{_{\sigma
^{-1}(k)}}],Q_{k})_{0}=Hom_{\Lambda }(S_{\sigma ^{-1}(k)}[-n_{_{\sigma
^{-1}(k)}}],I_{n}^{(k)})_{0}=Hom_{\Lambda }(S_{\sigma ^{-1}(k)}[-n_{_{\sigma
^{-1}(k)}}],I_{n}^{\prime (k)})_{0}=D(S_{k})[-n_{\sigma ^{-1}(k)}]$.

Therefore: $I_{n}^{\prime (k)}=D(Q_{\sigma ^{-1}(k)}^{\ast })[-n_{_{\sigma
^{-1}(k)}}]$.

Since $\Lambda $ has a decomposition $\Lambda \cong \underset{i=1}{\overset{m%
}{\oplus }}Q_{i}$, it has an injective resolution:

$0\rightarrow \Lambda \rightarrow I_{0}\rightarrow I_{1}\rightarrow
I_{2}\rightarrow ...I_{n-1}\rightarrow I_{n}\rightarrow I_{n+1}\rightarrow
...$ with $I_{j}=\underset{k=1}{\overset{m}{\oplus }}I_{j}^{(k)}$ and $%
I_{n}=I_{n}^{\prime }\oplus I_{n}^{\prime \prime }$, with $I_{n}^{\prime }=%
\underset{k=1}{\overset{m}{\oplus }}I_{n}^{\prime (k)}$ and $I_{n}^{\prime
\prime }=\underset{k=1}{\overset{m}{\oplus }}I_{n}^{\prime \prime (k)}$.

The module $I_{n}^{\prime \prime }$ is torsion free and $I_{n}^{\prime }=%
\underset{k=1}{\overset{m}{\oplus }}D(Q_{\sigma ^{-1}(k)}^{\ast
})[-n_{_{\sigma ^{-1}(k)}}]=\underset{k=1}{\overset{m}{\oplus }}%
D(Q_{j}^{\ast })[-n_{j}]$, after shifting.

We give $I_{n}^{\prime }$ a structure of $\Lambda $-$\Lambda $ bimodule as
follows:

Let $\lambda \in \Lambda $ be a non zero homogeneous element of $\Lambda $
and $\phi _{\lambda }$ the map $\phi _{\lambda }:\Lambda \rightarrow \Lambda 
$ given by right multiplication, $\phi _{\lambda }(x)=x\lambda $.

This map extends to a map of resolutions:

$%
\begin{array}{ccccccccccc}
0\rightarrow & \Lambda & \overset{\varepsilon }{\rightarrow } & I_{0} & 
\overset{d_{0}}{\rightarrow } & I_{1} & ... & I_{n} & \rightarrow & I_{n+1}
& ... \\ 
& \phi _{\lambda }\downarrow &  & \phi _{\lambda }^{0}\downarrow &  & \phi
_{\lambda }^{1}\downarrow &  & \phi _{\lambda }^{n}\downarrow &  & \phi
_{\lambda }^{n+1}\downarrow &  \\ 
0\rightarrow & \Lambda & \overset{\varepsilon }{\rightarrow } & I_{0} & 
\overset{d_{0}}{\rightarrow } & I_{1} & ... & I_{n} & \rightarrow & I_{n+1}
& ...%
\end{array}%
$

Since $I_{n}^{\prime \prime }$ is torsion free $Hom_{\Lambda }(I_{n}^{\prime
},I_{n}^{\prime \prime })=0$ and $\phi _{\lambda }^{n}:I_{n}^{\prime }\oplus
I_{n}^{\prime \prime }\rightarrow I_{n}^{\prime }\oplus I_{n}^{\prime \prime
}$ has triangular form: $\phi _{\lambda }^{n}=\left[ 
\begin{array}{cc}
\overset{\wedge }{\phi }_{\lambda } & \rho \\ 
0 & \sigma%
\end{array}%
\right] .$

We claim the map: $\overset{\wedge }{\phi }_{\lambda }:I_{n}^{\prime
}\rightarrow I_{n}^{\prime }$ is unique.

Assume $(\psi _{\lambda }^{0},$ $\psi _{0}^{1},\psi _{\lambda }^{2}...\psi
_{\lambda }^{n},\psi _{\lambda }^{n+1}...)$ is another lifting of $\phi
_{\lambda }$.

The map $\psi _{\lambda }^{n}$ has triangular form $\psi _{\lambda }^{n}=%
\left[ 
\begin{array}{cc}
\overset{\wedge }{\psi }_{\lambda } & \overline{\rho } \\ 
0 & \overline{\sigma }%
\end{array}%
\right] $

Then there exists a homotopy $\{s_{k}\}$, $s_{k}:I_{k}\rightarrow I_{k-1}$
such that $\phi _{\lambda }^{n}=d_{n-1}s_{n}+s_{n+1}d_{n}+\psi _{\lambda
}^{n}$.

There are decompositions: $s_{n}=(s_{n}^{\prime }$, $s_{n}^{\prime \prime })$%
, $d_{n}=(d_{n}^{\prime }$,$d_{n}^{\prime \prime })$, $s_{n+1}=\left[ 
\begin{array}{c}
s_{n+1}^{\prime } \\ 
s_{n+1}^{\prime \prime }%
\end{array}%
\right] $ and $d_{n-1}=\left[ 
\begin{array}{c}
d_{n-1}^{\prime } \\ 
d_{n-1}^{\prime \prime }%
\end{array}%
\right] .$

It follows: $d_{n-1}s_{n}+s_{n+1}d_{n}=\left[ 
\begin{array}{c}
d_{n-1}^{\prime } \\ 
d_{n-1}^{\prime \prime }%
\end{array}%
\right] (s_{n}^{\prime }$, $s_{n}^{\prime \prime })+\left[ 
\begin{array}{c}
s_{n+1}^{\prime } \\ 
s_{n+1}^{\prime \prime }%
\end{array}%
\right] (d_{n}^{\prime }$,$d_{n}^{\prime \prime }).$

Since $I_{n-1}$and $I_{n+1}$are torsion free, $s_{n}^{\prime }=0$ and $%
d_{n}^{\prime }=0$ and $d_{n-1}s_{n}+s_{n+1}d_{n}=\left[ 
\begin{array}{cc}
0 & d_{n-1}^{\prime }s_{n}^{\prime \prime }+s_{n+1}^{\prime }d_{n}^{\prime
\prime } \\ 
0 & d_{n-1}^{\prime \prime }s_{n}^{\prime \prime }+s_{n+1}^{\prime \prime
}d_{n}^{\prime \prime }%
\end{array}%
\right] .$

It follows: $\phi _{\lambda }^{n}=\left[ 
\begin{array}{cc}
\overset{\wedge }{\phi }_{\lambda } & \rho \\ 
0 & \sigma%
\end{array}%
\right] =$ $\left[ 
\begin{array}{cc}
\overset{\wedge }{\psi }_{\lambda } & \overline{\rho }+d_{n-1}^{\prime
}s_{n}^{\prime \prime }+s_{n+1}^{\prime }d_{n}^{\prime \prime } \\ 
0 & \overline{\sigma }+d_{n-1}^{\prime \prime }s_{n}^{\prime \prime
}+s_{n+1}^{\prime \prime }d_{n}^{\prime \prime }%
\end{array}%
\right] .$

Therefore: $\overset{\wedge }{\phi _{\lambda }}=\overset{\wedge }{\psi }%
_{\lambda }$.

\bigskip It is clear that there is a ring homomorphism $\Lambda \rightarrow
End_{\Lambda }(I_{n}^{\prime })$ given by $\lambda \rightarrow \overset{%
\wedge }{\phi _{\lambda }}$ which gives $I_{n}^{\prime }$ the structure of $%
\Lambda $-$\Lambda $ bimodule. Since $\Lambda e_{k}=Q_{k}$ with $e_{k}$ an
idempotent, multiplying by $e_{k}$ on the right is just the projection and
we have a commutative exact diagram:

$%
\begin{array}{ccccccccccc}
0\rightarrow & \Lambda & \overset{\varepsilon }{\rightarrow } & I_{0} & 
\overset{d_{0}}{\rightarrow } & I_{1} & ... & I_{n} & \rightarrow & I_{n+1}
& ... \\ 
& \phi _{e_{k}}\downarrow &  & \phi _{e_{k}}^{0}\downarrow &  & \phi
_{e_{k}}^{1}\downarrow &  & \phi _{e_{k}}^{n}\downarrow &  & \phi
_{e_{k}}^{n+1}\downarrow &  \\ 
0\rightarrow & \Lambda e_{k} & \rightarrow & I_{0}^{(k)} & \overset{d_{0}^{k}%
}{\rightarrow } & I_{1}^{(k)} & ... & I_{n}^{(k)} & \rightarrow & 
I_{n+1}^{(k)} & ...%
\end{array}%
$

where $\Lambda =\Lambda e_{k}\oplus \Lambda (1-e_{k}).$

Let $0\rightarrow \Lambda (1-e_{k})\rightarrow I_{0}^{(\ell )}\rightarrow
I_{1}^{(\ell )}\rightarrow I_{2}^{(\ell )}\rightarrow ...I_{n-1}^{(\ell
)}\rightarrow I_{n}^{(\ell )}\rightarrow I_{n+1}^{(\ell )}\rightarrow ...$

be the injective resolution of $\Lambda (1-e_{k})$ and $p_{j}:$ $%
I_{j}^{(k)}\oplus I_{j}^{(\ell )}\rightarrow I_{j}^{(k)}$ be the projection.
The $p_{j}=\phi _{e_{k}}^{j}.$

It is clear that $I_{n}=I_{n}^{\prime }\oplus I_{n}^{\prime \prime }$ and $%
I_{n}^{\prime }=I_{n}^{\prime (k)}\oplus I_{n}^{\prime (\ell )}$, $%
I_{n}^{\prime \prime }=I_{n}^{\prime \prime (k)}\oplus I_{n}^{\prime \prime
(\ell )}$. Therefore: $\overset{\wedge }{\phi }_{e_{k}}:I_{n}^{\prime
}\rightarrow I_{n}^{\prime (k)}$ is the projection.

Let $f:\Lambda \rightarrow \Lambda $ be a morphism of left $\Lambda $%
-modules and $j_{\ell }:\Lambda e_{\ell }\rightarrow \Lambda $, $%
p_{k}:\Lambda \rightarrow \Lambda e_{k}$ be the inclusion and the
projection, $f$ has matrix form $f=(f_{k\ell })$, with $f_{k\ell
}=p_{k}fj_{\ell }$. Any map $g:\Lambda e_{\ell }\rightarrow \Lambda e_{k}$
can be extended to a map $\overset{\wedge }{g}:I_{n}^{\prime (\ell
)}\rightarrow I_{n}^{\prime (k)}$ in particular $f_{k\ell }$ extends to a
map $\overset{\wedge }{f}_{k\ell }:I_{n}^{\prime (\ell )}\rightarrow
I_{n}^{\prime (k)}$and $f$ extends to the map $\overset{\wedge }{f}=(\overset%
{\wedge }{f}_{k\ell })$.

Consider now the opposite ring $\Lambda ^{op}.$

For each $\Lambda ^{op}$-module simple $S_{j}^{\prime }$ we have $%
Ext_{\Lambda ^{op}}^{n}(S_{j}^{\prime },\Lambda )=D(S_{\tau (j)}^{\prime
})[-m_{j}]$, but $Ext_{\Lambda }^{n}(S_{j},\Lambda )[n_{j}]=S_{j}^{\prime }$%
, hence; $Ext_{\Lambda ^{op}}^{n}(S_{j}^{\prime },\Lambda )=Ext_{\Lambda
^{op}}^{n}(Ext_{\Lambda }^{n}(S_{j},\Lambda )[n_{j}],\Lambda )=$\linebreak $%
Ext_{\Lambda ^{op}}^{n}(Ext_{\Lambda }^{n}(S_{j},\Lambda ),\Lambda
)[-n_{j}]=S_{j}[-n_{j}].$

Therefore: $m_{j}=n_{j}$ and $D(S_{\tau (j)}^{\prime })=S_{j}=S_{\sigma \tau
(j)}.$It follows $\tau =\sigma ^{-1}.$

In a similar way we have a right injection resolution:

$0\rightarrow Q_{k}^{\ast }=e_{k}\Lambda \rightarrow J_{0}^{(k)}\rightarrow
J_{1}^{(k)}\rightarrow J_{2}^{(k)}\rightarrow ...J_{n-1}^{(k)}\rightarrow
J_{n}^{(k)}\rightarrow J_{n+1}^{(k)}\rightarrow ...$

and $\Lambda $ has a right injective resolution:

$0\rightarrow \Lambda \rightarrow \overset{m}{\underset{k=1}{\oplus }}%
J_{0}^{(k)}\rightarrow \overset{m}{\underset{k=1}{\oplus }}%
J_{1}^{(k)}\rightarrow \overset{m}{\underset{k=1}{\oplus }}%
J_{2}^{(k)}\rightarrow ...\overset{m}{\underset{k=1}{\oplus }}%
J_{n-1}^{(k)}\rightarrow \overset{m}{\underset{k=1}{\oplus }}%
J_{n}^{(k)}\rightarrow \overset{m}{\underset{k=1}{\oplus }}J_{n+1}^{(k)}...$

$J_{n}=J_{n}^{\prime }\oplus J_{n}^{\prime \prime }$ with $J_{n}^{\prime }$
torsion and $J_{n}^{\prime \prime }$ torsion free $J_{n}^{\prime }=\overset{m%
}{\underset{k=1}{\oplus }}J_{n}^{\prime (k)}$, $J_{n}^{\prime \prime }=$ $%
\overset{m}{\underset{k=1}{\oplus }}J%
{\acute{}}%
{\acute{}}%
_{n}^{(k)}.$

Similarly to the left sided situation, $J_{n}^{\prime (k)}=D(Q_{\sigma
(k)})[-n_{\sigma (k)}]$.

Given a map $g:\Lambda e_{\ell }\rightarrow \Lambda e_{k}$, there is a map $%
\overset{\wedge }{g}:$D(Q$_{\tau (\ell )}^{\ast }$)[-n$_{\tau (\ell )}$]$%
\rightarrow $D(Q$_{\tau (k)}^{\ast }$)[-n$_{\tau (k)}$], hence a map $(D$($%
\overset{\wedge }{g}))^{\ast }:P_{\tau (\ell )}[-n_{\tau (\ell
)}]\rightarrow P_{\tau (k)}[-n_{\tau (k)}].$

Write $\tau (g)=$ $(D$($\overset{\wedge }{g}))^{\ast }$.

We want to prove that $\tau $ is an equivalence from the category of
indecomposable projective left $\Lambda $-modules concentrated in degree
zero, to the category of indecomposable projective $\{Q_{j}[-n_{j}]\}$ $%
\Lambda $-modules with shift.

Given a map $f:e_{k}\Lambda \rightarrow e_{\ell }\Lambda $, given by $%
f(e_{k})=e_{\ell }\lambda e_{k}$, there is a map $\overset{\wedge }{f}%
:D(Q_{\sigma (k)})[-n_{\sigma (k)}]\rightarrow $ $D(Q_{\sigma (\ell
)})[-n_{\sigma (\ell )}]$ and ($D($ $\overset{\wedge }{f}))^{\ast
}:Q_{\sigma (k)}^{\ast }[-n_{\sigma (k)}]\rightarrow Q_{\sigma (\ell
)}^{\ast }[-n_{\sigma (\ell )}]$.

We write $\sigma (f)=$($D($ $\overset{\wedge }{f}))^{\ast }$.

We claim $f=(\tau ((\sigma (f))^{\ast })^{\ast }$.

The multiplication map $\mu :e_{k}\Lambda \otimes _{\Lambda }J_{n}^{\prime
}\rightarrow e_{k}J_{n}^{\prime }$ is an isomorphism. Let $\overline{f}=\mu
(f\otimes 1)\mu ^{-1}:e_{k}J_{n}^{\prime }\rightarrow e_{k}J_{n}^{\prime }$
be the induced map. Then $\overset{\wedge }{f}(e_{k}y)=e_{\ell }\lambda
e_{k}y=\overline{f}(e_{k}y).$Therefore: $\overset{\wedge }{f}=D((\sigma
(f)^{\ast })=\overline{f}.$

By definition of $\mu $, the diagram: $%
\begin{array}{ccc}
Q_{k}^{\ast }\otimes J_{n}^{\prime } & \overset{\mu }{\rightarrow } & 
e_{k}J_{n}^{\prime } \\ 
\downarrow f\otimes 1 &  & \downarrow \overline{f} \\ 
Q_{\ell }^{\ast }\otimes J_{n}^{\prime } & \overset{\mu }{\rightarrow } & 
e_{\ell }J_{n}^{\prime }%
\end{array}%
$, commutes.

\bigskip Let $M$ be a left $\Lambda $-module of finite length. It follows by
induction that $X=Hom_{\Lambda }(M,I_{n}^{\prime })$, in fact $%
X=Ext_{\Lambda }^{n}(M,\Lambda )$.

We have a chain of isomorphisms:

$Q_{k}^{\ast }\otimes M\cong Ext_{\Lambda ^{op}}^{n}(Ext_{\Lambda
}^{n}(M,\Lambda ),Q_{k}^{\ast })\cong Q_{k}^{\ast }\otimes Hom_{\Lambda
^{op}}(X,J_{n}^{\prime })\cong e_{k}Hom_{\Lambda ^{op}}(X,J_{n}^{\prime })$%
\linebreak $\cong Hom_{\Lambda ^{op}}(X,e_{k}J_{n}^{\prime })\cong
Hom_{\Lambda ^{op}}(X,Q_{k}^{\ast }\otimes J_{n}^{\prime })$.

We have a commutative diagram:

\begin{center}
*) $%
\begin{array}{ccc}
\text{Q}_{k}^{\ast }\otimes \text{Hom}_{\Lambda ^{op}}\text{(X,J}%
_{n}^{\prime }\text{)}\cong & \text{Hom}_{\Lambda ^{op}}\text{(X,e}_{k}\text{%
J}_{n}^{\prime }\text{)}\cong & \text{Hom}_{\Lambda ^{op}}\text{(X,D(Q}%
_{\sigma (k)}\text{))[-n}_{\sigma (k)}\text{]} \\ 
\downarrow \text{f}\otimes \text{1} & \downarrow \text{(X,f)} & \downarrow 
\text{(X,D((}\sigma \text{(f)}^{\ast }\text{))} \\ 
\text{Q}_{\ell }^{\ast }\otimes \text{Hom}_{\Lambda ^{op}}\text{(X,J}%
_{n}^{\prime }\text{)}\cong & \text{Hom}_{\Lambda ^{op}}\text{(X,e}_{\ell }%
\text{J}_{n}^{\prime }\text{)}\cong & \text{Hom}_{\Lambda ^{op}}\text{(X,D(Q}%
_{\sigma (\ell )}\text{))[-n}_{\sigma (\ell )}\text{]}%
\end{array}%
$
\end{center}

We also have in the opposite ring a commutative diagram:

\begin{center}
**) $%
\begin{array}{ccc}
\text{Hom}_{\Lambda }\text{(M,I}_{n}^{\prime }\text{)}\otimes Q_{\sigma
(\ell )}\cong & \text{Hom}_{\Lambda }\text{(M,I}_{n}^{\prime }e_{\sigma
(\ell )}\text{)}\cong & \text{Hom}_{\Lambda }\text{(M,D(Q}_{\tau \sigma
(\ell )}^{\ast }\text{))[-n}_{\ell }\text{]} \\ 
\downarrow \text{1}\otimes \sigma (f)^{\ast } & \downarrow \text{(M,}%
\overline{\sigma (\text{f)}^{\ast }} & \downarrow \text{(M,D((}\tau \text{(}%
\sigma \text{(f)}^{\ast }\text{))}^{\ast }\text{)} \\ 
\text{Hom}_{\Lambda }\text{(M,I}_{n}^{\prime }\text{)}\otimes Q_{\sigma
(k)}\cong & \text{Hom}_{\Lambda }\text{(M,I}_{n}^{\prime }e_{\sigma (k)}%
\text{)}\cong & \text{Hom}_{\Lambda }\text{(M,D(Q}_{\tau \sigma (k)}^{\ast }%
\text{))[-n}_{k}\text{]}%
\end{array}%
$
\end{center}

Dualizing $\ast )$ we get the commutative diagram:

\begin{center}
$%
\begin{array}{cc}
\text{D(Hom}_{\Lambda ^{op}}\text{(X,D(Q}_{\sigma (\ell )}\text{)))[n}%
_{\sigma (\ell )}\text{]}\cong & \text{D(Q}_{\ell }^{\ast }\otimes \text{Hom}%
_{\Lambda ^{op}}\text{(X,J}_{n}^{\prime }\text{))} \\ 
\downarrow \text{D(X,D((}\sigma \text{(f)}^{\ast })\text{)} & \downarrow 
\text{D(f}\otimes \text{1)} \\ 
\text{D(Hom}_{\Lambda ^{op}}\text{(X,D(Q}_{\sigma (k)}\text{)))[n}_{\sigma
(k)}\text{]}\cong & \text{D(Q}_{k}^{\ast }\otimes \text{Hom}_{\Lambda ^{op}}%
\text{(X,J}_{n}^{\prime }\text{))}%
\end{array}%
$
\end{center}

and using the isomorphism: Hom$_{\Lambda ^{op}}$(X,D(Q$_{r}$))$\cong $D(X$%
\otimes $Q$_{r}$) we get the commutative diagrams:

\begin{center}
$%
\begin{array}{cc}
\text{D(Hom}_{\Lambda ^{op}}\text{(X,D(Q}_{\sigma (\ell )}\text{)))[n}%
_{\sigma (\ell )}\text{]} & \text{X}\otimes \text{Q}_{\sigma (\ell )}\text{[n%
}_{\sigma (\ell )}\text{]} \\ 
\downarrow \text{D(X,D((}\sigma \text{(f)}^{\ast })\text{)} & \downarrow 
\text{1}\otimes \sigma \text{(f)}^{\ast } \\ 
\text{D(Hom}_{\Lambda ^{op}}\text{(X,D(Q}_{\sigma (k)}\text{)))[n}_{\sigma
(k)}\text{]} & \text{X}\otimes \text{Q}_{\sigma (k)}\text{[n}_{\sigma (k)}%
\text{]}%
\end{array}%
$
\end{center}

\bigskip

\begin{center}
$%
\begin{array}{cc}
\text{D(Q}_{\ell }^{\ast }\otimes \text{Hom}_{\Lambda ^{op}}\text{(X,J}%
_{n}^{\prime }\text{))}\cong & \text{Hom}_{\Lambda }\text{(Hom}_{\Lambda
^{op}}\text{(X,J}_{n}^{\prime }\text{),D(Q}_{\ell }^{\ast }\text{))} \\ 
\downarrow \text{D(f}\otimes \text{1)} & \downarrow \text{(Hom}_{\Lambda
^{op}}\text{(X,J}_{n}^{\prime }\text{),D(f))} \\ 
\text{D(Q}_{k}^{\ast }\otimes \text{Hom}_{\Lambda ^{op}}\text{(X,J}%
_{n}^{\prime }\text{))}\cong & \text{Hom}_{\Lambda }\text{(Hom}_{\Lambda
^{op}}\text{(X,J}_{n}^{\prime }\text{),D(Q}_{k}^{\ast }\text{))}%
\end{array}%
$
\end{center}

Therefore: there is a commutative diagram:

\begin{center}
$%
\begin{array}{cc}
\text{X}\otimes \text{Q}_{\sigma (\ell )}\text{[n}_{\sigma (\ell )}\text{]}%
\cong & \text{Hom}_{\Lambda }\text{(Hom}_{\Lambda ^{op}}\text{(X,J}%
_{n}^{\prime }\text{),D(Q}_{\ell }^{\ast }\text{))} \\ 
\downarrow \text{1}\otimes \sigma \text{(f)}^{\ast } & \downarrow \text{(Hom}%
_{\Lambda ^{op}}\text{(X,J}_{n}^{\prime }\text{),D(f))} \\ 
\text{X}\otimes \text{Q}_{\sigma (k)}\text{[n}_{\sigma (k)}\text{]}\cong & 
\text{Hom}_{\Lambda }\text{(Hom}_{\Lambda ^{op}}\text{(X,J}_{n}^{\prime }%
\text{),D(Q}_{k}^{\ast }\text{))}%
\end{array}%
$
\end{center}

From the last diagram and diagram: **) we obtain D(f))=D(($\tau $($\sigma $%
(f)$^{\ast }$))$^{\ast }$). Therefore: f=($\tau $($\sigma $(f)$^{\ast }$))$%
^{\ast }$, as claimed.

Similarly we get $g=(\sigma ((\tau (g)^{\ast }))^{\ast }$.

It follows from these equalities that $\sigma $ is an equivalence:

Assume $\sigma (f)=\sigma (g)$. Then $\tau ($ $\sigma (f)^{\ast })^{\ast
}=\tau (\sigma (g)^{\ast })^{\ast }$ implies $f=g$.

If $h:Q_{k}^{\ast }\rightarrow Q_{\ell }^{\ast }$ is a map, let $f$ be the
map $f=(\tau (h^{\ast }))^{\ast }$. Then $(\sigma (\tau (h^{\ast }))^{\ast
})^{\ast }=h^{\ast }=\sigma (f)^{\ast }$.

It follows $\sigma (f)=h$.

We have proved $\sigma $, $\tau $ are equivalences.

Then there is a commutative triangle:

\begin{center}
$%
\begin{array}{ccc}
\Lambda \cong End_{\Lambda }(\Lambda )^{op} & \overset{\alpha }{\rightarrow }
& End_{\Lambda }(\overset{m}{\underset{k=1}{\oplus }}D(Q_{\tau (k)}^{\ast
}))[-n_{\tau (k)}])^{op} \\ 
& \searrow \tau & \downarrow (D(-))^{\ast } \\ 
&  & End_{\Lambda }(\overset{m}{\underset{k=1}{\oplus }}Q_{\tau
(k)})[-n_{\tau (k)}])^{op}%
\end{array}%
$
\end{center}

Where $\alpha ((f_{ij}))=(\overset{\wedge }{f}_{ij})$. Since $D(-))^{\ast }$%
and $\tau $ are isomorphism, $\alpha $ is also an isomorphism.

Applying the duality we have a ring isomorphism: $\Lambda \cong End_{\Lambda
^{op}}(D(\Lambda ))\rightarrow End_{\Lambda ^{op}}(\overset{m}{\underset{k=1}%
{\oplus }}Q_{\tau (k)}^{\ast })[n_{\tau (k)}])$.

In order to continue the study of the bimodule structure of $I_{n}^{\prime }$
we need the following:

\begin{lemma}
Let $\Lambda $ be a positively graded locally finite $\Bbbk $-algebra, $E$ a
finitely cogenerated injective left $\Lambda $-module and $\{M_{\alpha },\pi
_{\alpha }\}_{\alpha \in A}$ an inverse system of locally finite graded left 
$\Lambda $-modules such that $M=\underleftarrow{\lim }M_{\alpha }$ is
locally finite. Then we have a natural isomorphism: $Hom_{\Lambda }(%
\underleftarrow{\lim }M_{\alpha },E)=\underrightarrow{\lim }$ $Hom_{\Lambda
}(M_{\alpha },E)$.
\end{lemma}

\begin{proof}
Since $E$ is finitely cogenerated it is locally finite. It was proved in $%
[12]$ that $D$ is a duality in the category of locally finite graded
modules. Applying the duality and adjunction we have natural isomorphisms:

Hom$_{\Bbbk }(Hom_{\Lambda }(\underleftarrow{\lim }M_{\alpha },E),\Bbbk
)\cong D(E)\otimes \underleftarrow{\lim }M_{\alpha }$.

Using the fact $D(E)$ is a finitely generated projective, we have
isomorphisms:

$D(E)\otimes \underleftarrow{\lim }M_{\alpha }\cong Hom_{\Lambda
}(D(E)^{\ast }$, $\underleftarrow{\lim }M_{\alpha })\cong \underleftarrow{%
\lim }Hom_{\Lambda }(D(E)^{\ast }$, $M_{\alpha })\cong $\linebreak $%
\underleftarrow{\lim }(D(E)\otimes M_{\alpha })\cong \underleftarrow{\lim }%
D(Hom_{\Lambda }(M_{\alpha },E)).$

It follows from $[7]$ that $\underleftarrow{\lim }D(Hom_{\Lambda }(M_{\alpha
},E))\cong D(\underrightarrow{\lim }Hom_{\Lambda }(M_{\alpha },E)).$

Dualizing again we obtain the isomorphism: Hom$_{\Lambda }$($\underleftarrow{%
\lim }$M$_{\alpha }$,E)$\cong \underrightarrow{\lim }$Hom$_{\Lambda }$(M$%
_{\alpha }$,E).
\end{proof}

As a corollary we have:

\begin{proposition}
Let $\Lambda $ be a graded AS Gorenstein algebra. With the same notation as
above there is an isomorphism of left (right) $\Lambda $-modules $%
I_{n}^{\prime }\cong J_{n}^{\prime }$.
\end{proposition}

\begin{proof}
Let $\varphi :\Lambda \rightarrow Hom_{\Lambda }(I_{n}^{\prime
},I_{n}^{\prime })^{op}$ be the ring isomorphism given above.

The bimodule structure of $I_{n}^{\prime }$ gives $Hom_{\Lambda
}(I_{n}^{\prime },I_{n}^{\prime })$ a structure of right module as follows:
given a map $f:I_{n}^{\prime }\rightarrow I_{n}^{\prime }$ and $\lambda $ in 
$\Lambda $, $f\lambda (x)=f(x)\lambda .$

Then $\varphi (\lambda _{1}\lambda _{2})=\varphi (\lambda _{1})\ast \varphi
(\lambda _{2})=\varphi (\lambda _{2})\varphi (\lambda _{1})$

$\varphi (\lambda _{1}\lambda _{2})(x)=x(\lambda _{1}\lambda _{2})=\varphi
(\lambda _{1})(x)\lambda _{2}=\varphi (\lambda _{1})\lambda _{2}(x)$.

Therefore: $\varphi (\lambda _{1}\lambda _{2})=\varphi (\lambda _{1})\lambda
_{2}$ and $\varphi $ is an isomorphism of right modules.

Since $I_{n}^{\prime }$ is finitely cogenerated $I_{n}^{\prime }=%
\underrightarrow{\lim }$ $I_{\alpha }^{\prime }$ with $I_{\alpha }^{\prime }$
of finite length and \linebreak $Hom_{\Lambda ^{op}}(Hom_{\Lambda
}(I_{\alpha }^{\prime },I_{n}^{\prime }),J_{n}^{\prime })\cong I_{\alpha
}^{\prime }.$

We have natural isomorphisms:

$Hom_{\Lambda ^{op}}(Hom_{\Lambda }(I_{n}^{\prime },I_{n}^{\prime
}),J_{n}^{\prime })\cong Hom_{\Lambda ^{op}}(Hom_{\Lambda }(\underrightarrow{%
\lim }I_{\alpha }^{\prime },I_{n}^{\prime }),J_{n}^{\prime })\cong $

$Hom_{\Lambda ^{op}}(\underleftarrow{\lim }Hom_{\Lambda }(I_{\alpha
}^{\prime },I_{n}^{\prime }),J_{n}^{\prime })\cong \underrightarrow{\lim }%
Hom_{\Lambda ^{op}}(Hom_{\Lambda }(I_{\alpha }^{\prime },I_{n}^{\prime
}),J_{n}^{\prime })\cong \underrightarrow{\lim }I_{\alpha }^{\prime }\cong
I_{n}^{\prime }.$

Also $Hom_{\Lambda ^{op}}(Hom_{\Lambda }(I_{n}^{\prime },I_{n}^{\prime
}),J_{n}^{\prime })\cong Hom_{\Lambda ^{op}}(\Lambda ^{op},J_{n}^{\prime
})\cong J_{n}^{\prime }$.

It follows $I_{n}^{\prime }$ is isomorphic to $J_{n}^{\prime }$ as right $%
\Lambda $-module.

Going to the opposite ring and interchanging the roles of $I_{n}^{\prime }$
and $J_{n}^{\prime }$we have that $J_{n}^{\prime }$ is isomorphic to $%
I_{n}^{\prime }$ as left $\Lambda $-module.
\end{proof}

\begin{corollary}
With the same conditions as in the proposition we get an isomorphisms of
left (right) $\Lambda $-modules $\overset{m}{\underset{k=1}{\oplus }}Q_{\tau
(k)}^{\ast })[n_{\tau (k)}]\cong \overset{m}{\underset{k=1}{\oplus }}%
Q_{\sigma (k)})[n_{\sigma (k)}].$
\end{corollary}

\begin{proof}
We just apply duality to the isomorphic modules $I_{n}^{\prime }$ and $%
J_{n}^{\prime }$ to get:

$\overset{m}{\underset{k=1}{\oplus }}Q_{\tau (k)}^{\ast })[n_{\tau
(k)}]\cong \overset{m}{\underset{k=1}{\oplus }}Q_{\sigma (k)})[n_{\sigma
(k)}].$
\end{proof}

\begin{lemma}
Let $\Lambda $ be a positively graded locally finite $\Bbbk $-algebra, $X$, $%
Y$ graded left $\Lambda $-modules with $Y$ locally finite. Then for any non
negative integer $n$, there is a natural isomorphism: $Ext_{\Lambda
}^{n}(X,Y)\cong D(Tor_{n}^{\Lambda }(D(Y),X)).$
\end{lemma}

\begin{proof}
Consider a projective resolution of $X$:

$\rightarrow Q_{n+1}\rightarrow Q_{n}\rightarrow Q_{n-1}...\rightarrow
Q_{1}\rightarrow Q_{0}\longrightarrow X\rightarrow 0.$

Tensoring with $D(Y)$ we obtain a complex:

*) $\rightarrow D(Y)\otimes _{\Lambda }Q_{n+1}\rightarrow D(Y)\otimes
_{\Lambda }Q_{n}\rightarrow D(Y)\otimes _{\Lambda }Q_{n-1}...\rightarrow
D(Y)\otimes _{\Lambda }Q_{1}\rightarrow D(Y)\otimes _{\Lambda
}Q_{0}\longrightarrow 0.$

whose $n$-th homology is $Tor_{n}^{\Lambda }(D(Y),X)$.

Dualizing *) we obtain a complex:

**) $0\rightarrow D(D(Y)\otimes _{\Lambda }Q_{0})\rightarrow D(D(Y)\otimes
_{\Lambda }Q_{1})\rightarrow ...D(D(Y)\otimes _{\Lambda }Q_{n-1}\rightarrow
D(D(Y)\otimes _{\Lambda }Q_{n})\rightarrow D(D(Y)\otimes _{\Lambda
}Q_{n+1})\longrightarrow ...$

with $n$-th homology $D(Tor_{n}^{\Lambda }(D(Y),X)$ $)$.

Since we have natural isomorphisms: $D(D(Y)\otimes _{\Lambda }Q_{i})\cong
Hom_{\Lambda }(Q_{i}$,$D^{2}(Y))\cong Hom_{\Lambda }(Q_{i}$, $Y)$, the
complex **) is isomorphic to the complex:

$0\rightarrow Hom_{\Lambda }(Y,Q_{0})\rightarrow Hom_{\Lambda
}(Y,Q_{1})\rightarrow ...Hom_{\Lambda }(Y,Q_{n-1})\rightarrow Hom_{\Lambda
}(Y,Q_{n})\rightarrow Hom_{\Lambda }(Y,Q_{n+1})\longrightarrow ...$

whose $n$-th homology $Ext_{\Lambda }^{n}(X,Y)$.

It follows: $Ext_{\Lambda }^{n}(X,Y)\cong D(Tor_{n}^{\Lambda }(D(Y),X)).$
\end{proof}

\begin{corollary}
Let $\Lambda $ be a graded AS Gorenstein algebra of graded injective
dimension $n$. Then for any graded $\Lambda $-module $M$, there is a natural
isomorphism:

$Ext_{\Lambda }^{i}(M$,$\overset{m}{\underset{k=1}{\oplus }}Q_{\sigma
(k)})[n_{\sigma (k)}])\cong D(Tor_{i}^{\Lambda }(I_{n}^{\prime }$, $M))$.
\end{corollary}

\begin{corollary}
Let $\Lambda $ be a positively graded locally finite $\Bbbk $- algebra, $\{$ 
$X_{\alpha },j_{\alpha }\}$ be a graded direct system of $\Lambda $-modules
and $Y$ a locally finite $\Lambda $-module. Then for any non negative
integer $i$, there is a natural isomorphism: $Ext_{\Lambda }^{i}(%
\underrightarrow{\lim }$ $X_{\alpha },Y)\cong \underleftarrow{\lim }%
Ext_{\Lambda }^{i}(X_{\alpha },Y)$.
\end{corollary}

\begin{proof}
From the isomorphism: $Tor_{i}^{\Lambda }(D(Y),\underrightarrow{\lim }%
X_{\alpha })\cong \underrightarrow{\lim }Tor_{i}^{\Lambda }(D(Y),X_{\alpha
}) $, we obtain a chain of isomorphisms $Ext_{\Lambda }^{i}(\underrightarrow{%
\lim }$ $X_{\alpha },Y)\cong D(Tor_{i}^{\Lambda }(D(Y),\underrightarrow{\lim 
}$ $X_{\alpha })\cong D(\underrightarrow{\lim }Tor_{i}^{\Lambda
}(D(Y),X_{\alpha }))\cong \underleftarrow{\lim }D(Tor_{i}^{\Lambda
}(D(Y),X_{\alpha }))\cong $ $\underleftarrow{\lim }Ext_{\Lambda
}^{i}(X_{\alpha },Y)$.
\end{proof}

As an application of these observations we have the following:

Let $\Lambda $ be a graded Gorenstein algebra, such that all graded simple
have projective resolutions consisting of finitely generated projective, $m$
be the graded radical of $\Lambda $ and $k$ a positive integer. Then there
are isomorphisms:

$Ext_{\Lambda }^{n}(\Lambda /\Lambda _{\geq k},\Lambda )\cong Hom_{\Lambda
}(\Lambda /\Lambda _{\geq k},I_{n}^{\prime })$ and $D(Hom_{\Lambda }(\Lambda
/\Lambda _{\geq k},I_{n}^{\prime }))\cong D(I_{n}^{\prime })\otimes \Lambda
/\Lambda _{\geq k}$.

Dualizing again and using adjunction:

$Hom_{\Lambda }(\Lambda /\Lambda _{\geq k},I_{n}^{\prime })\cong
Hom_{\Lambda ^{op}}(D(I_{n}^{\prime }),D(\Lambda /\Lambda _{\geq k})).$

Using the fact $D(I_{n}^{\prime })$ is a finitely generated projective, $%
Hom_{\Lambda }(\Lambda /\Lambda _{\geq k},I_{n}^{\prime })\cong D(\Lambda
/\Lambda _{\geq k})\otimes D(I_{n}^{\prime })^{\ast }$.

It follows: $\underrightarrow{\lim }Ext_{\Lambda }^{n}(\Lambda /\Lambda
_{\geq k},\Lambda )\cong \underrightarrow{\lim }(D(\Lambda /\Lambda _{\geq
k})\otimes D(I_{n}^{\prime })^{\ast })\cong (\underrightarrow{\lim }%
(D(\Lambda /\Lambda _{\geq k}))\otimes D(I_{n}^{\prime })^{\ast }\cong D(%
\underleftarrow{\lim }\Lambda /\Lambda _{\geq k})\otimes D(I_{n}^{\prime
})^{\ast }\cong D(\Lambda )\otimes D(I_{n}^{\prime })^{\ast }\cong
Hom_{\Lambda ^{op}}(D(I_{n}^{\prime }),D(\Lambda ))\cong Hom_{\Lambda
}(\Lambda ,I_{n}^{\prime })\cong I_{n}^{\prime }.$

Observe that the functor $\underrightarrow{\lim }Hom_{\Lambda }(\Lambda
/\Lambda _{\geq k},-)$ is left exact and it has left derived functors $%
\underrightarrow{\lim }Ext_{\Lambda }^{i}(\Lambda /\Lambda _{\geq k},-)$.

\begin{definition}
Let $\Lambda $ be a positively graded locally finite $\Bbbk $-algebra and $m$
the graded Jacobson radical. The $i$-th local cohomology of the module $M$
is $\underrightarrow{\lim }Ext_{\Lambda }^{i}(\Lambda /\Lambda _{\geq k},M)$.
\end{definition}

We write $\Gamma _{m}(M)=\underrightarrow{\lim }Hom_{\Lambda }(\Lambda
/\Lambda _{\geq k},M)$ and $\Gamma _{m}^{i}(M)=\underrightarrow{\lim }%
Ext_{\Lambda }^{i}(\Lambda /\Lambda _{\geq k},M)$. We have proved above that
for a graded AS Gorenstein algebra of graded injective dimension $n$, , such
that all graded simple have projective resolutions consisting of finitely
generated projective, $\Gamma _{m}^{n}(\Lambda )=I_{n}^{\prime }$.

The results of Proposition 4, can be improved using the following:[9]

\begin{lemma}
$\Lambda $ , $\Gamma $ be a positively graded $\Bbbk $-algebras such that
the graded simple have projective resolutions consisting of finitely
generated projective modules, $m_{\Lambda }$, $m_{\Gamma }$, the graded
Jacobson radicals of $\Lambda $ and $\Gamma $, respectively and $m_{\Lambda
\otimes _{\Bbbk }\Gamma }$ the graded Jacobson radical of $\Lambda \otimes
_{\Bbbk }\Gamma $. Then $\Gamma _{m_{\Lambda \otimes _{\Bbbk }\Gamma }}$=$%
\Gamma _{m_{\Gamma }}\circ \Gamma _{m_{\Lambda }}$=$\Gamma _{m_{\Lambda
}}\circ \Gamma _{m_{\Gamma }}$
\end{lemma}

\begin{proof}
From inequalities: $(\Lambda \otimes _{\Bbbk }\Gamma )_{\geq k}\supseteq
\Lambda _{0}\otimes _{\Bbbk }\Gamma _{\geq k}$+$\Lambda _{\geq k}\otimes
_{\Bbbk }\Gamma _{0}\supseteq (\Lambda \otimes _{\Bbbk }\Gamma )_{\geq 2k}$
, it follows:\newline
$\underrightarrow{\lim }Hom_{\Lambda \otimes \Gamma }(\Lambda \otimes \Gamma
/((\Lambda \otimes \Gamma )_{\geq k},-)$=$\underrightarrow{\lim }%
Hom_{\Lambda \otimes \Gamma }(\Lambda \otimes \Gamma /(\Lambda _{0}\otimes
_{\Bbbk }\Gamma _{\geq k}$+$\Lambda _{\geq k}\otimes _{\Bbbk }\Gamma
_{0}),-) $\newline
$=\underrightarrow{\lim }Hom_{\Lambda \otimes \Gamma }(\Lambda /\Lambda
_{\geq k}\otimes \Gamma /\Gamma _{\geq k}$,-)=$\underrightarrow{\lim }%
Hom_{\Gamma }(\Gamma /\Gamma _{\geq k},$ $Hom_{\Lambda }(\Lambda /\Lambda
_{\geq k},-))=$\newline
$\underrightarrow{\lim }\underrightarrow{\lim }Hom_{\Gamma }(\Gamma /\Gamma
_{\geq k},$ $Hom_{\Lambda }(\Lambda /\Lambda _{\geq k},-))$=$%
\underrightarrow{\lim }Hom_{\Gamma }(\Gamma /\Gamma _{\geq k},$ $%
\underrightarrow{\lim }Hom_{\Lambda }(\Lambda /\Lambda _{\geq k},-))$\newline
$=\Gamma _{m_{\Gamma }}\circ \Gamma _{m_{\Lambda }}$.

Also $\underrightarrow{\lim }Hom_{\Lambda \otimes \Gamma }(\Lambda /\Lambda
_{\geq k}\otimes \Gamma /\Gamma _{\geq k}$,-)=$\underrightarrow{\lim }%
Hom_{\Lambda }(\Lambda /\Lambda _{\geq k}$,$Hom_{\Gamma }(\Gamma /\Gamma
_{\geq k},-))$=\newline
$\underrightarrow{\lim }\underrightarrow{\lim }Hom_{\Lambda }(\Lambda
/\Lambda _{\geq k}$, $Hom_{\Gamma }(\Gamma /\Gamma _{\geq k},-))$=$%
\underrightarrow{\lim }Hom_{\Lambda }(\Lambda /\Lambda _{\geq k}$, $%
\underrightarrow{\lim }Hom_{\Gamma }(\Gamma /\Gamma _{\geq k},$-))\linebreak
=$\Gamma _{m\Lambda }\circ \Gamma _{m\Gamma }$.
\end{proof}

The following result can be found in [9], for completeness we include it
here.

\begin{proposition}
Let $\Lambda $ be a positively graded $\Bbbk $-algebras such that the graded
simple have projective resolutions consisting of finitely generated
projective, $m$ the graded radical of $\Lambda $ and $m^{op}$ the graded
radical of $\Lambda ^{op}$. Then for any integer $k$, $\Gamma
_{m}^{k}(\Lambda )=\Gamma _{m^{op}}^{k}(\Lambda ).$
\end{proposition}

\begin{proof}
Let $0\rightarrow \Lambda \rightarrow E_{0}\rightarrow E_{1}\rightarrow
...E_{t}\rightarrow E_{t+1}\rightarrow ...$ be an injective resolution of $%
\Lambda $ as bimodule, it is easy to prove that each $E_{t}$ is injective
both as left and as right module and let $\mathfrak{E}$ be the complex: $%
0\rightarrow E_{0}\rightarrow E_{1}\rightarrow ...E_{t}\rightarrow
E_{t+1}\rightarrow ...$

Then we have:

$\Gamma _{m_{\Lambda \otimes _{\Bbbk }\Lambda ^{op}}}(\mathfrak{E)}=\Gamma
_{m^{op}}\circ \Gamma _{m}(\mathfrak{E})=\underrightarrow{\lim }Hom_{\Lambda
^{op}}(\Lambda /\Lambda _{\geq k},$ $\underrightarrow{\lim }Hom_{\Lambda
}(\Lambda /\Lambda _{\geq \ell },\mathfrak{E}))=$

$\underrightarrow{\lim }\underrightarrow{\lim }Hom_{\Lambda ^{op}}(\Lambda
/\Lambda _{\geq k},$ $Hom_{\Lambda }(\Lambda /\Lambda _{\geq \ell },%
\mathfrak{E}))=\underrightarrow{\lim }\underrightarrow{\lim }Hom_{\Lambda
}(\Lambda /\Lambda _{\geq \ell }\otimes _{\Lambda }\Lambda /\Lambda _{\geq
k},\mathfrak{E})$

$=\underrightarrow{\lim }\underrightarrow{\lim }Hom_{\Lambda }(\Lambda
/\Lambda _{\geq \ell }+\Lambda _{\geq k},\mathfrak{E})=\underrightarrow{\lim 
}Hom_{\Lambda }(\Lambda /\Lambda _{\geq \ell },\mathfrak{E})=\Gamma _{m}(%
\mathfrak{E}).$

Similarly, $\Gamma _{m_{\Lambda \otimes _{\Bbbk }\Lambda ^{op}}}(\mathfrak{%
E)=}\Gamma _{m^{op}}(\mathfrak{E}).$

Taking the $k$-th homology of the complex we obtain $\Gamma _{m}^{k}(\Lambda
)=\Gamma _{m^{op}}^{k}(\Lambda ).$
\end{proof}

As an application of this equality we have the following:

\begin{corollary}
For a graded AS Gorenstein algebra $\Lambda $ of graded injective dimension $%
n$, such that all graded simple have projective resolutions consisting of
finitely generated projective, $I_{n}^{\prime }\cong J_{n}^{\prime }$ as a
bimodule.
\end{corollary}

\section{Local Cohomology}

The aim of this section is to prove the Local Cohomology formula for a class
of graded Gorenstein algebras. We already have all the ingredients to prove
it for graded Gorenstein algebras of finite local cohomology dimension and
such that all graded simple have graded projective resolutions consisting of
finitely generated projective modules.

\begin{definition}
We say that a graded AS Gorenstein algebra $\Lambda $ has finite left local
cohomology dimension, if there is a non negative integer $\ell _{0}$ such
that for any left $\Lambda $-module $M$ and any integer $\ell >\ell _{0}$, $%
\Gamma _{m}^{\ell }(M)=0$.
\end{definition}

Let $\Lambda $ be an algebra satisfying these conditions, consider the
projection maps $\overline{\pi }_{k}:\Lambda /\Lambda _{\geq k}\rightarrow
\Lambda /$ $\Lambda _{\geq k-1}$, the collection \{$\overline{\pi }_{k}$,$%
\Lambda /\Lambda _{\geq k}$\}$_{k\geq 0}$ forms a graded inverse system.

For each $k$, there is a minimal graded projective resolution: $\mathfrak{P}%
^{(k)}\rightarrow \Lambda /\Lambda _{\geq k}$.

$\mathfrak{P}^{(k)}$:... $P_{n+1}^{(k)}\rightarrow P_{n}^{(k)}\rightarrow
P_{n-1}^{(k)}\rightarrow ...P_{1}^{(k)}\rightarrow P_{0}^{(k)}\rightarrow
\Lambda /\Lambda _{\geq k}\rightarrow 0$.

The projections $\pi _{n+1}^{(k)}:\Lambda /\Lambda _{\geq k}\rightarrow
\Lambda /\Lambda _{\geq k-1}$ induce maps of projective resolutions:

\begin{center}
$%
\begin{array}{ccccccc}
\text{...P}_{n+1}^{(k)} & \rightarrow \text{P}_{n}^{(k)} & \rightarrow \text{%
P}_{n-1}^{(k)} & \rightarrow \text{...P}_{1}^{(k)} & \rightarrow \text{P}%
_{0}^{(k)} & \rightarrow \Lambda \text{/}\Lambda _{\geq k} & \rightarrow 
\text{0} \\ 
\downarrow \pi _{n+1}^{(k)} & \downarrow \pi _{n}^{(k)} & \downarrow \pi
_{n-1}^{(k)} & \downarrow \pi _{1}^{(k)} & \downarrow \pi _{0}^{(k)} & 
\downarrow \overline{\pi }_{k} &  \\ 
\text{...P}_{n+1}^{(k-1} & \rightarrow \text{P}_{n}^{(k-1)} & \rightarrow 
\text{P}_{n-1}^{(k-1)} & \rightarrow \text{...P}_{1}^{(k-1)} & \rightarrow 
\text{P}_{0}^{(k-1)} & \rightarrow \Lambda \text{/}\Lambda _{\geq k-1} & 
\rightarrow \text{0}%
\end{array}%
$
\end{center}

Dualizing with respect to the ring we get a map of complexes:

$%
\begin{array}{ccccccc}
\mathfrak{P}^{\ast (k-1)}\text{:} & \text{0}\rightarrow \text{P}_{0}^{\ast
(k-1)} & \rightarrow \text{P}_{1}^{\ast (k-1)} & \rightarrow \text{P}%
_{2}^{\ast (k-1)} & \text{...P}_{n}^{\ast (k-1)} & \rightarrow \text{P}%
_{n+1}^{\ast (k-1)} & \text{...} \\ 
\downarrow \text{(}\pi _{k}\text{)}^{\ast } & \downarrow \text{(}\pi
_{0}^{(k)}\text{)}^{\ast } & \downarrow \text{(}\pi _{1}^{(k)}\text{)}^{\ast
} & \downarrow \text{(}\pi _{2}^{(k)}\text{)}^{\ast } & \downarrow \text{(}%
\pi _{n}^{(k)}\text{)}^{\ast } & \downarrow \text{(}\pi _{n+1}^{(k)}\text{)}%
^{\ast } &  \\ 
\mathfrak{P}^{\ast (k)}\text{:} & \text{0}\rightarrow \text{P*}_{0}^{(k)} & 
\rightarrow \text{P}_{1}^{\ast (k)} & \rightarrow \text{P}_{2}^{\ast (k)} & 
\text{...P}_{n}^{\ast (k)} & \rightarrow \text{P}_{n+1}^{\ast (k)} & \text{%
...}%
\end{array}%
$

where $H^{i}(\mathfrak{P}^{\ast (k)})=Ext_{\Lambda }^{i}(\Lambda /\Lambda
_{\geq k},\Lambda )=0$ for $i\neq n$ and $H^{n}(\mathfrak{P}^{\ast (k)})=$%
\linebreak $Ext_{\Lambda }^{n}(\Lambda /\Lambda _{\geq k},\Lambda )$ is
different from zero and of finite length.

In addition, for any graded left $\Lambda $-module $M$, $\mathfrak{P}^{\ast
(k)}\otimes M\cong Hom_{\Lambda }(\mathfrak{P}^{(k)},M).$

Hence; $H^{i}(\mathfrak{P}^{\ast (k)}\otimes M)\cong Ext_{\Lambda
}^{i}(\Lambda /\Lambda _{\geq k},M)$.

We have a direct system of complexes \{$\mathfrak{P}^{\ast (k)}$, ($\pi _{k}$%
)$^{\ast }\}$ and $\mathfrak{F}=\underrightarrow{\lim }\mathfrak{P}^{\ast
(k)}$ is a complex of flat modules. Since $\underrightarrow{\lim }$ is an
exact functor, $H^{i}(\mathfrak{F)}=H^{i}(\underrightarrow{\lim }\mathfrak{P}%
^{\ast (k)})=\underrightarrow{\lim }H^{i}(\mathfrak{P}^{\ast (k)})=%
\underrightarrow{\lim }Ext_{\Lambda }^{i}(\Lambda /\Lambda _{\geq k},\Lambda
)$, and $\underrightarrow{\lim }(\mathfrak{P}^{\ast (k)}\otimes M)\cong 
\underrightarrow{\lim }Hom_{\Lambda }(\mathfrak{P}^{(k)},M).$

Therefore: $H^{i}($ $\underrightarrow{\lim }(\mathfrak{P}^{\ast (k)}\otimes
M))\cong H^{i}(\underrightarrow{\lim }Hom_{\Lambda }(\mathfrak{P}%
^{(k)},M))\cong $

$\underrightarrow{\lim }H^{i}(Hom_{\Lambda }(\mathfrak{P}^{(k)},M))\cong 
\underrightarrow{\lim }Ext_{\Lambda }^{i}(\Lambda /\Lambda _{\geq k},M)$=$%
\Gamma _{m}^{i}(M)$.

We see next that the assumption of finite cohomological dimension $\ell _{0}$
imposes strong restrictions on the complex $\mathfrak{F}.$

$\mathfrak{F}:$ $0\rightarrow F_{0}\overset{d_{0}}{\rightarrow }F_{1}\overset%
{d_{1}}{\rightarrow }F_{2}\overset{d_{2}}{\rightarrow }...F_{j}...%
\rightarrow F_{\ell _{0}}\overset{d_{\ell _{0}}}{\rightarrow }F_{\ell _{0}+1}%
\overset{d_{\ell _{0}+1}}{\rightarrow }F_{\ell _{0}+2}\rightarrow ...$

The complex $\mathfrak{F}$ has homology $H^{i}(\mathfrak{F)}=%
\underrightarrow{\lim }Ext_{\Lambda }^{i}(\Lambda /\Lambda _{\geq k},\Lambda
)=0$ for $i\neq n$. Since $\Gamma _{m}^{n}(\Lambda )=I_{n}^{\prime }\neq 0$, 
$\ell _{0}\geq n$ and $H^{i}(\mathfrak{F}\otimes M)=0$ for $i>\ell _{0}$.

Consider the exact sequence: *) $F_{\ell _{0}}\overset{d_{\ell _{0}}}{%
\rightarrow }F_{\ell _{0}+1}\overset{d_{\ell _{0}+1}}{\rightarrow }F_{\ell
_{0}+2}\overset{d_{\ell _{0}+2}}{\rightarrow }C\rightarrow 0.$

*) is part of a flat resolution of $C$, tensoring with a graded left $%
\Lambda $-module $M$ we obtain a complex: $F_{\ell _{0}}\otimes M\overset{%
d_{\ell _{0}}\otimes 1}{\rightarrow }F_{\ell _{0}+1}\otimes M\overset{%
d_{\ell _{0}+1}\otimes 1}{\rightarrow }F_{\ell _{0}+2}\otimes M\rightarrow 0$%
, where $Kerd_{\ell _{0}+1}\otimes 1/\func{Im}d_{\ell _{0}}\otimes 1=\Gamma
_{m}^{\ell _{0}+1}(\mathfrak{F}\otimes M)=Tor_{1}^{\Lambda }(C,M)=0$ for all
graded $\Lambda $-modules $M$.

This implies $C$ is flat.

Consider the exact sequence: $0\rightarrow C_{1}\rightarrow F_{\ell _{0}+2}%
\overset{d_{\ell _{0}+2}}{\rightarrow }C\rightarrow 0.$ By the long homology
sequence there is an exact sequence:

\begin{center}
$Tor_{2}^{\Lambda }(C,M)\rightarrow Tor_{1}^{\Lambda }(C_{1},M)\rightarrow
Tor_{1}^{\Lambda }(F_{\ell _{0}+2},M)$
\end{center}

Since $C$ and $F_{\ell _{0}+2}$ are flat, it follows $Tor_{1}^{\Lambda
}(C_{1},M)=0$ for all $M$, hence $C_{1}$ is also flat.

We use induction to get an exact sequence: $0\rightarrow Kerd_{n}\rightarrow
F_{n}\rightarrow \func{Im}d_{n}\rightarrow 0$ with $F_{n}$ and $\func{Im}%
d_{n}$ flat. It follows $F_{n}^{\prime }=Kerd_{n}$ is flat and we have a
complex $\mathfrak{F}^{\prime }$ of flat modules:

$\mathfrak{F}^{\prime }:$ $0\rightarrow F_{0}\overset{d_{0}}{\rightarrow }%
F_{1}\overset{d_{1}}{\rightarrow }F_{2}\overset{d_{2}}{\rightarrow }%
...F_{j}...\rightarrow F_{n-1}\overset{d_{n-1}}{\rightarrow }F_{n}^{\prime
}\rightarrow 0$ such that $\Gamma ^{i}(\mathfrak{F}^{\prime })=\Gamma ^{i}(%
\mathfrak{F)=}0$ for $i\neq n$ and $\Gamma ^{n}(\mathfrak{F}^{\prime }%
\mathfrak{)=}\Gamma ^{n}(\mathfrak{F)=}\underrightarrow{\lim }Ext_{\Lambda
}^{n}(\Lambda /\Lambda _{\geq k},\Lambda )=I_{n}^{\prime }.$

Also for any graded left $\Lambda $-module $M$, the sequence: $0\rightarrow $
$Tor_{1}^{\Lambda }(\func{Im}d_{n},M)\rightarrow F_{n}^{\prime }\otimes
M\rightarrow F_{n}\otimes M\rightarrow \func{Im}d_{n}\otimes M\rightarrow 0$
is exact, and $\func{Im}d_{n}$ flat implies $0\rightarrow F_{n}^{\prime
}\otimes M\rightarrow F_{n}\otimes M\rightarrow \func{Im}d_{n}\otimes
M\rightarrow 0$ is exact. Similarly, the sequence: $0\rightarrow \func{Im}%
d_{n}\otimes M\rightarrow F_{n+1}\otimes M\rightarrow \func{Im}%
d_{n+1}\otimes M\rightarrow 0$ is exact and the sequence: $0\rightarrow
F_{n}^{\prime }\otimes M\rightarrow F_{n}\otimes M\rightarrow F_{n-1}\otimes
M$ is exact.

It follows $\Gamma ^{n}(\mathfrak{F}^{\prime }\otimes M)=\Gamma ^{n}(%
\mathfrak{F}\otimes M)=\Gamma ^{n}(M)$.

Since $\mathfrak{F}^{\prime }\rightarrow I_{n}^{\prime }$ is a flat
resolution of $I_{n}^{\prime }$. The complex $\mathfrak{F}^{\prime }\otimes
M $ has $i$-th homology $Tor_{n-i}^{\Lambda }(I_{n}^{\prime },M).$ It
follows $Tor_{n-i}^{\Lambda }(I_{n}^{\prime },M)=\underrightarrow{\lim }%
Ext_{\Lambda }^{i}(\Lambda /\Lambda _{\geq k},M)$=$\Gamma _{m}^{i}(M)$.
Applying the duality we obtain the local cohomology formula:

$D(\underrightarrow{\lim }Ext_{\Lambda }^{i}(\Lambda /\Lambda _{\geq
k},M))=Ext_{\Lambda }^{n-i}(M,$ $D(\Gamma _{m}^{n}(\Lambda ))$ , for all $%
0\leq i\leq n$.

We have proved the following:

\begin{theorem}
Let $\Lambda $ be a graded AS Gorenstein algebra of graded injective
dimension $n$ and such that all graded simple modules have projective
resolutions consisting of finitely generated projective modules and assume $%
\Lambda $ has finite local cohomology dimension. Then for any graded left
module $M$ there is a natural isomorphism: $D(\underrightarrow{\lim }%
Ext_{\Lambda }^{i}(\Lambda /\Lambda _{\geq k},M))=Ext_{\Lambda }^{n-i}(M,$ $%
D(\Gamma _{m}^{n}(\Lambda ))$ , for $0\leq i\leq n$.
\end{theorem}

\bigskip The theorem generalizes results for Artin-Schelter regular algebras
proved in [11].

We will finish the paper giving a family of examples of graded AS Gorenstein
algebras of finite local cohomology dimension.

Let $\Lambda =\Bbbk Q_{1}/I_{1}$ be a graded selfinjective non semisimple
algebra and $\Gamma =\Bbbk Q_{2}/I_{2}$ an Artin-Schelter regular algebra of
global dimension $n$, in the sense of $[12],[15]$, denote by $r$, $m$, the
graded Jacobson radicals of $\Lambda $ and $\Gamma $, respectively and $%
\Lambda _{0}=\Lambda /r$, $\Gamma _{0}=\Gamma /m.$

The following formula was proved in $[13]\ $:

$Ext_{\Lambda \otimes \Gamma }^{t}(\Lambda _{0}\otimes _{\Bbbk }\Gamma
_{0},\Lambda \otimes _{K}\Gamma )\cong \underset{i+j=m}{\oplus }Ext_{\Lambda
}^{i}(\Lambda _{0},\Lambda )\otimes _{\Bbbk }Ext_{\Gamma }^{j}(\Gamma
_{0},\Gamma ).$

Since $\Lambda $ is selfinjective $Ext_{\Lambda \otimes \Gamma }^{t}(\Lambda
_{0}\otimes _{\Bbbk }\Gamma _{0},\Lambda \otimes _{\Bbbk }\Gamma )\cong
Hom_{\Lambda }(\Lambda _{0},\Lambda )\otimes _{\Bbbk }Ext_{\Gamma
}^{t}(\Gamma _{0},\Gamma )$ and $\Gamma $ Artin-Schelter regular implies $%
Ext_{\Gamma }^{t}(\Gamma _{0},\Gamma )=0$ for $t\neq 0$ and $Ext_{\Gamma
}^{n}(\Gamma _{0},\Gamma )=\Gamma _{0}^{op}[-n]$.

It follows $Ext_{\Lambda \otimes \Gamma }^{t}(\Lambda _{0}\otimes _{\Bbbk
}\Gamma _{0},\Lambda \otimes _{\Bbbk }\Gamma )=0$ for $t\neq n$ and $%
Ext_{\Lambda \otimes \Gamma }^{n}(\Lambda _{0}\otimes _{\Bbbk }\Gamma
_{0},\Lambda \otimes _{\Bbbk }\Gamma )=\Lambda _{0}^{op}\otimes _{\Bbbk
}\Gamma _{0}^{op}[-n]$.

We have proved $\Lambda \otimes _{\Bbbk }\Gamma $ is graded Gorenstein of
injective dimension $n$.

The graded radical of $\Lambda \otimes _{\Bbbk }\Gamma $ is $\mathfrak{m}$=$%
\Lambda \otimes m+r\otimes \Gamma $. Then $\mathfrak{m}^{k}=\underset{i+j=k}{%
\sum }r^{i}\otimes m^{j}.$ Assume $r^{t-1}\neq 0$ and $r^{t}=0$. For $k\geq
t $, $\mathfrak{m}^{k}=($ $\underset{i+j=t-1}{\sum }r^{i}\otimes
m^{j})\Lambda \otimes m^{k-t+1}\subseteq (\Lambda \otimes _{\Bbbk }\Gamma
)\Lambda \otimes m^{k-t+1}=\Lambda \otimes m^{k-t+1}$.

We have inequalities: $\Lambda \otimes m^{k}\subseteq \mathfrak{m}%
^{k}\subseteq \Lambda \otimes m^{k-t+1}\subseteq m^{k-t+1}.$ Hence, we have
surjective maps: $\Lambda \otimes _{\Bbbk }\Gamma /$ $\Lambda \otimes
m^{k}\rightarrow \Lambda \otimes _{\Bbbk }\Gamma /\mathfrak{m}%
^{k}\rightarrow \Lambda \otimes _{\Bbbk }\Gamma /\Lambda \otimes
m^{k-t+1}\rightarrow \Lambda \otimes _{\Bbbk }\Gamma /m^{k-t+1}$, which
induce maps:

$\rightarrow Ext_{\Lambda \otimes \Gamma }^{j}(\Lambda \otimes _{\Bbbk
}\Gamma /\mathfrak{m}^{k-t+1},M)\rightarrow Ext_{\Lambda \otimes \Gamma
}^{j}(\Lambda \otimes _{\Bbbk }(\Gamma /m^{k-t+1}),M)\rightarrow
Ext_{\Lambda \otimes \Gamma }^{j}($ $\Lambda \otimes _{\Bbbk }\Gamma /$ $%
\mathfrak{m}^{k},M)\rightarrow Ext_{\Lambda \otimes \Gamma }^{j}(\Lambda
\otimes _{\Bbbk }(\Gamma /m^{k}),M).$

Taking direct limits the sequence:

$\rightarrow \underrightarrow{\lim }Ext_{\Lambda \otimes \Gamma
}^{j}(\Lambda \otimes _{\Bbbk }\Gamma /\mathfrak{m}^{k-t+1},M)\rightarrow 
\underrightarrow{\lim }Ext_{\Lambda \otimes \Gamma }^{j}(\Lambda \otimes
_{\Bbbk }(\Gamma /m^{k-t+1}),M)\rightarrow $\linebreak $\underrightarrow{%
\lim }Ext_{\Lambda \otimes \Gamma }^{j}($ $\Lambda \otimes _{\Bbbk }\Gamma /$
$\mathfrak{m}^{k},M)\rightarrow \underrightarrow{\lim }Ext_{\Lambda \otimes
\Gamma }^{j}(\Lambda \otimes _{\Bbbk }(\Gamma /m^{k}),M)$

is exact.

Using the fact: $\underrightarrow{\lim }Ext_{\Lambda \otimes \Gamma
}^{j}(\Lambda \otimes _{\Bbbk }\Gamma /\mathfrak{m}^{k-t+1},M)=%
\underrightarrow{\lim }Ext_{\Lambda \otimes \Gamma }^{j}($ $\Lambda \otimes
_{\Bbbk }\Gamma /$ $\mathfrak{m}^{k},M)$ and $\underrightarrow{\lim }%
Ext_{\Lambda \otimes \Gamma }^{j}(\Lambda \otimes _{\Bbbk }(\Gamma
/m^{k-t+1}),M)=\underrightarrow{\lim }Ext_{\Lambda \otimes \Gamma
}^{j}(\Lambda \otimes _{\Bbbk }(\Gamma /m^{k}),M)$, it follows:

$\underrightarrow{\lim }Ext_{\Lambda \otimes \Gamma }^{j}($ $\Lambda \otimes
_{\Bbbk }\Gamma /$ $\mathfrak{m}^{k},M)\cong \underrightarrow{\lim }%
Ext_{\Lambda \otimes \Gamma }^{j}(\Lambda \otimes _{\Bbbk }(\Gamma
/m^{k}),M).$

Let $0\rightarrow Q_{n}^{(k)}\rightarrow Q_{n-1}^{(k)}...\rightarrow
Q_{1}^{(k)}\rightarrow \Gamma \longrightarrow \Gamma /m^{k}\rightarrow 0$ be
a graded projective resolution of $\Gamma /m^{k}$.

Then $0\rightarrow \Lambda \otimes Q_{n}^{(k)}\rightarrow \Lambda \otimes
Q_{n-1}^{(k)}...\rightarrow \Lambda \otimes Q_{1}^{(k)}\rightarrow \Lambda
\otimes \Gamma \longrightarrow \Lambda \otimes \Gamma /m^{k}\rightarrow 0$
is a graded projective resolution of $\Lambda \otimes _{\Bbbk }(\Gamma
/m^{k})$ and $Ext_{\Lambda \otimes \Gamma }^{j}(\Lambda \otimes _{\Bbbk
}(\Gamma /m^{k}),M)=0$ for $j>n$. It follows $H_{\mathfrak{m}}^{j}(M)=%
\underrightarrow{\lim }Ext_{\Lambda \otimes \Gamma }^{j}($ $\Lambda \otimes
_{\Bbbk }\Gamma /$ $\mathfrak{m}^{k},M)=0$ for $j>n$.

We have proved $\Lambda \otimes _{\Bbbk }\Gamma $ has local cohomology
dimension $n$.

We give two concrete examples in which this situation natural arises:

1) Let $\Lambda =\Bbbk <X_{1},X_{2},...X_{n}>/\{X_{i}^{2},$ $%
X_{i}X_{j}+X_{j}X_{i}$ \}$_{i\neq j}$ be the exterior algebra in $n$%
-variables and $\Gamma =\Bbbk \lbrack X_{1},X_{2},...X_{n}]$ the polynomial
algebra. Then the ring of polynomial forms $\Gamma \otimes \Lambda $ is
Gorenstein of finite local cohomology dimension $n$.

This example appears as the cohomology ring of the group algebra of an
elementary abelian $p$ group, over a field of characteristic $p>2$. [5]

2) Let $Q$ be a non Dynkin quiver with only sinks and sources $\Bbbk $ a
field, $\Lambda =\Bbbk Q\trianglerighteq D(\Bbbk Q)$ the trivial extension
and $\Gamma $ the preprojective algebra [10] corresponding to the quiver $Q.$
Then $\Gamma \otimes \Lambda $ is Gorenstein of finite local cohomology
dimension $2$.

\begin{center}
{\LARGE References}
\end{center}

\bigskip

[1] Artin M., Schelter W. Graded algebras of global dimension 3, Adv. Math.
66 (1987), 171-216.

[2]Auslander M. Reiten I. k-Gorenstein algebras and syzygy nodules. Journal
of Pure and Applied Algebra 92 (1994) 1-27.

[3] Auslander M. Reiten I. Cohen-Maculay and Gorenstein artin algebras,
Progress in Math. 95, 1991 Birkh\"{a}user, 221-245.

[4] Bruns W. Herzog J. Cohen-Macaulay rings, Cambridge studies in advanced
mathematics 39, 1993.

[5] Carlson J.F. The Varieties and the Cohomology Ring of a Module, J. of
Algebra, Vol. 85, No. 1, (1983) 104-143.

[6] Fossum R.M. Griffith P.A., Reiten I. Trivial Extensions of Abelian
Categories, Lecture Notres in Math. \$56 Springer, 1975.

[7] Iwanaga Y. Sato H. On Auslander n-Gorenstein rings. J. Pure Applied
Algebra. (1996) 61-76.

[8] J$\varnothing $rgensen P. Properties of As Cohen-Macaulay algebras,. J.
Pure Apl. Algebra !38 (1999), 239-249.

[9] J$\varnothing $rgensen P. Local Cohomology for Non Commutative Graded
Algebras. Comm. in Algebra, 25(2), 575-591 (1997)

[10] Mart\'{\i}nez-Villa, R. Applications of Koszul algebras: the
preprojective algebra. Representation theory of algebras (Cocoyoc, 1994),
487--504, CMS Conf. Proc., 18, Amer. Math. Soc., Providence, RI, 1996.

[11] Mart\'{\i}nez-Villa, R. Serre duality for generalized Auslander regular
algebras. Trends in the representation theory of finite-dimensional algebras
(Seattle, WA, 1997), 237--263, Contemp. Math., 229, Amer. Math. Soc.,
Providence, RI, 1998.

[12] Martinez-Villa, R. Graded, Selfinjective, and Koszul Algebras, J.
Algebra 215, 34-72 1999

[13] Martinez-Villa, R. Koszul algebras and the Gorenstein condition.
Representations of algebras (S\~{a}o Paulo, 1999), 135--156, Lecture Notes
in Pure and Appl. Math., 224, Dekker, New York, 2002.

[14]Mart\'{\i}nez-Villa R. Solberg O. Graded and Koszul Categories, Appl.
Categ. Structures 18 (2010), no. 6, 615--652.

[15] Mart\'{\i}nez-Villa R. Solberg O. Artin Schelter regular algebras and
categories, J. Pure Appl. Algebra 215 (2011), no. 4, 546--565

[16] Miyashita Y. Tilting modules of Finite Projective Dimension. Math. Z.
193, 113-146 (1986)

[17] Rotman J.J. An Introduction to Homological Algebra, Second Edition
Universitext, Springer, 2009.

[18] Van Den Bergh M. Existence theorems for dualizing complexes over
non-commutative graded and filtered rings. J. Algebra 195 (1997), no.2,
662-679.

[19] Zaks, A. Injective dimension of semiprimary rings, J. of Algebra 13
(1969), 390-398.

[20] Zhang J.J. Connected Graded Gorenstein Algebras with Enough Normal
Elements, J. Algebra 189 (1997), 390-405.

\end{document}